\documentclass[
a4paper,				
11pt,
english
]{extarticle}

\newif\ifshow
\showtrue   

\makeatletter
\title{On the Role of the Double Fourier Sphere Method in Fast Algorithms on SO(3)}\let\Title\@title
\newcommand{\Date}{June 7, 2026}
\date{Institute for Applied Analysis, TU Bergakademie Freiberg \\ \Date}
\makeatother
\newcommand{\keywords}{Fast algorithms,Double Fourier Sphere,Rotation group,Harmonic analysis,Wigner transform,Crystallographic texture analysis}

\usepackage{header}
\usepackage{headerFormulas}

\usepackage[final]{microtype}	

\usepackage{soul}
\usepackage[normalem]{ulem}
\usepackage[most]{tcolorbox}

\newif\ifshowchanges
\showchangestrue
\showchangesfalse 

\ifshowchanges
  \newcommand{\added}[1]{{\sethlcolor{green!25}\hl{#1}}}
  \newcommand{\deleted}[1]{\textcolor{red}{\sout{#1}}}
  \newcommand{\changed}[1]{{\sethlcolor{yellow!40}\hl{#1}}}
  \newcommand{\changedmath}[1]{\mathhighlight{yellow!40}{#1}}
  \newtcolorbox{addedbox}{breakable, colback=green!25, colframe=green!25, boxrule=0.4pt, arc=1pt, left=4pt, right=4pt, top=4pt, bottom=4pt}
  \newtcolorbox{changedbox}{breakable, colback=yellow!40, colframe=yellow!40, boxrule=0.4pt, arc=1pt, left=4pt, right=4pt, top=4pt, bottom=4pt}
\else
  \newcommand{\added}[1]{#1}
  \newcommand{\deleted}[1]{}
  \newcommand{\changed}[1]{#1}
  \newcommand{\changedmath}[1]{#1}
  \newenvironment{addedbox}{}{}
  \newenvironment{changedbox}{}{}
\fi

\begin{document}
\setcounter{section}{0}
\setcounter{subsection}{1}

\author{
  Ralf Hielscher\orcidlink{0000-0002-6342-1799}%
  \thanks{ralf.hielscher@math.tu-freiberg.de}~,~%
  Erik Wünsche \orcidlink{0000-0001-5178-6654}%
  \thanks{erik.wuensche@math.tu-freiberg.de}}

\maketitle
\begin{abstract}
  We analyze the Double Fourier Sphere (DFS) method on the rotation group $\SO3$ in the frequency domain and demonstrate its central role in fast algorithms.
  Fast Fourier algorithms on $\SO3$ are commonly formulated as a Wigner transform - mapping harmonic to Fourier coefficients - followed by a Fourier transform.
  We revisit this formulation and interpret the Wigner transform as an explicit realization of the DFS method, lifting functions from $\SO3$ to $\IT^{3}$. In this context, we analyze the Sobolev regularity loss induced by this lifting.
  Furthermore, we compare different Wigner transform implementations, examine additional symmetry enhancements, and observe that the direct method is often faster and more stable than the fast polynomial transform approaches.
\end{abstract}



\section{Introduction}

Functions on the rotation group $\SO3$ arise naturally across many areas of science and engineering. Prominent examples include robotics~\cite{Chirikjian2001} and computer vision~\cite{Makadia2003}, as well as protein docking in bioinformatics~\cite{Kovacs2003}. A particularly important application stems from crystallography in geology and materials science, where orientation density functions describe the distribution of crystal orientations in polycrystalline materials. Such density functions on $\SO3$ play a central role in understanding and predicting the macroscopic behavior of materials~\cite{Boogaart2007,Hielscher2010}.

From a numerical point of view, it is desirable to have efficient and accurate methods to approximate, manipulate, and transform such functions. The Matlab toolbox \texttt{MTEX} (Mathematical Texture Analysis)~\cite{Hielscher2007} provides a high-level framework for analyzing and visualizing functions on $\SO3$. In particular, the algorithms presented in this paper are implemented and validated within this framework.

Just as Fourier expansions are fundamental on the torus $\mathbb T^{3}$, they play an equally important role on $\SO3$. Consequently, functions on $\SO3$ are well suited for Fourier-type expansions and for the application of fast Fourier algorithms. Harmonic series expansions therefore provide a natural and convenient framework for numerical analysis on $\SO3$. The efficient evaluation of such expansions, and in particular the fast computation of the $\SO3$ Fourier transform, has motivated decades of research.

The idea to exploit this structure in crystallography goes back to Hans Joachim Bunge in 1969, who established the use of harmonic series expansions on $\SO3$ as a standard tool for texture analysis~\cite{Bunge1982}.
In this context, he introduced the concept of converting harmonic expansions of Wigner-D functions into ordinary Fourier series on the torus, by substituting the Fourier expansion of the Wigner-D functions.
This procedure is commonly referred to as the Wigner transform, since it computes the Fourier coefficients from the harmonic coefficients.
In a subsequent step, fast Fourier techniques can be applied to the resulting Fourier series. For instance, the nonequispaced fast Fourier transform (NFFT) enables efficient evaluation at arbitrary sample points.
Bunge’s realization of the Wigner transform had a computational complexity of $\complexity{N^{4}}$, where $N$ denotes the bandwidth of the harmonic series expansion. However, his work was limited by the computational resources of his time, which restricted practical computations to very low bandwidths.
The approach was revitalized in 1996 by Risbo~\cite{Risbo1996}, whose work is still frequently cited for establishing a practically feasible formulation of the Wigner transform.

Later, in 2008, Potts~\cite{Potts2009} and Kostelec~\cite{Kostelec2008} independently developed faster algorithms for the Wigner transform, reducing the complexity to $\complexity{N^{3}\log^{2} N}$ by employing fast polynomial transform techniques. While \changed{this} method improves the asymptotic complexity, it requires costly precomputations.
In this paper, we will demonstrate that, in practice, this algorithm tends to be slower and less stable for moderate bandwidths than Bunge's original direct Wigner algorithm.

In 2023, a different perspective was introduced by Mildenberger and Quellmalz~\cite{Mildenberger2023}.
They proposed a generalized Double Fourier Sphere (DFS) framework for approximating functions on certain classes of manifolds, including the rotation group $\SO3$, by means of Fourier series on the three-torus~$\IT^{3}$.
In this approach, a function $f\colon\SO3\to\IC$ is first transformed to a function $g=f\circ\phi_{\SO3}\colon\IT^3\to\IC$ via the Euler angle parametrization $\phi_{\SO3}\colon\IT^{3}\to\SO3$.
The transformed function is then approximated by a Fourier series on $\IT^{3}$.
Exploiting symmetry properties induced by the parametrization, Mildenberger and Quellmalz lifted a Fourier basis of a subspace of $\Lp{2}(\IT^{3})$ to $\L2SO3$.
Although the parametrization $\phi_{\SO3}$ is not a diffeomorphism, they analyzed how \changed{Hölder-smoothness} is preserved under the DFS method and established uniform convergence together with explicit approximation rates depending on the smoothness of the function.

In contrast, we start from the harmonic basis of $\L2SO3$ consisting of the Wigner-D functions and effectively express their pullback to $\IT^{3}$ with respect to the Fourier basis by describing the $\SO3$ DFS operator in Fourier space.
Thus, while previous work analyzed the DFS method primarily in the spatial domain and focused on smoothness preservation, our approach explicitly characterizes the transformation in frequency space.
We show, that the DFS method and the Wigner transform are essentially the same, one in the spatial domain and the other one in the frequency domain.

\begin{addedbox}
  This allows us to quantify the loss of regularity on the Sobolev scale.
  The counterexample in \Cref{lem:CounterExampleSobolev1/2} shows that the loss of regularity is at least $\frac12$.
  In \Cref{Cor:SobolevRegularityLoss}, we establish $\frac34$ as an upper threshold for the loss of regularity.
  Although we are not able to close this gap, numerical experiments indicate that the lower bound $\frac12$ might be sharp.
\end{addedbox}

This paper is organized as follows: In \Cref{sec:Preliminaries}, we introduce the framework of harmonic series on $\SO3$, establishing the notation and key concepts.
Building on this, \Cref{sec:DFS} investigates the DFS method on $\SO3$ in the frequency domain. We first consider band-limited functions and then extend the analysis to non-band-limited functions. In this context we determine the Sobolev regularity required for the DFS transform to lie in $\Lp{2}(\IT^{3})$.
In \Cref{sec:NFSOFT}, we exploit the $\SO3$-Fourier transform and its adjoint in the context of the Wigner transform.
We show how appropriate quadrature schemes allow for efficient computation of harmonic coefficients and demonstrate how common crystallographic symmetries can be incorporated to reduce both storage and computational cost of the Wigner transform.
Finally, \Cref{sec:Numerics} presents a numerical comparison of both algorithmic realizations of the Wigner transform, illustrating the practical trade-offs between efficiency and stability.


\section{Preliminaries}\label{sec:Preliminaries}

The special orthogonal group in $\IR^3$ also known as rotation group
\begin{equation*}
  \SO3 = \set{\mat R \in \IR^{3\times3} }{\mat R^T\mat R=\id \text{ and } \det(\mat R)=1}
\end{equation*}
is a three-dimensional Riemannian manifold endowed with the structure of a compact Lie group.
A rotation $\mat R_{\eta}(\omega) \in \SO3$ can be parametrized by a rotation axis $\eta\in\S{2}$ and a rotation angle $\omega \in \IT=\IR/(2\pi\IZ)$.
An equivalent representation is given in terms of Euler angles $\alpha\in\IT$, $\beta\in[0,\pi]$ and $\gamma\in\IT$, where
\begin{equation*}
  \mat R(\alpha,\beta,\gamma) = \mat R_z(\alpha) \, \mat R_y(\beta) \, \mat R_z(\gamma) \in\SO3.
\end{equation*}
Since $\SO3$ is a compact Lie group, it admits a unique, bi-invariant Haar measure $\mu$, normalized with respect to the total volume of $\SO3$.
In Euler angle coordinates, the measure reads
\[ \d{\mu\big(\mat R(\alpha,\beta,\gamma) \big)} = \frac1{8\pi^2}\sin\beta\d{\alpha}\d{\beta}\d{\gamma}. \]
This yields the Hilbert space $\L2SO3$ with inner product and norm
\begin{align*}
\scp{f}{g}_{\L2SO3} \coloneqq \int_{\SO3} f(\mat R) \, \conj{g(\mat R)} \d{\mu(\mat R)}, \quad  \|f\|_{\L2SO3} \coloneqq \sqrt{\scp{f}{f}_{\L2SO3}}
\end{align*}
for arbitrary $f,g\in\L2SO3$.

\begin{definition}
  Let $n \in \IN_0$ and $k,l \in \IZ$ with $|k|,|l| \leq n$.

  The $\Lp{2}$-normalized Wigner-D functions~$\WignerD{n}{k}{l} \colon \SO3 \to \IC$ of degree $n$ and orders $k,l$ are defined as
  \begin{equation*}
    \WignerD{n}{k}{l}(\mathbf{R}(\alpha,\beta,\gamma)) \coloneqq \sqrt{2n+1} \, \e^{-\i \, k \, \alpha} \, \Wignerd{n}{k}{l}(\cos\beta) \, \e^{-\i \, l \, \gamma},
  \end{equation*}
  see~\cite{Varshalovich1988}.
  Here, $\Wignerd{n}{k}{l} \colon [-1,1] \to \IR$ denotes the Wigner-d function, defined by
  \begin{equation*}
    \Wignerd{n}{k}{l}(x) \coloneqq (-1)^{\nu} \binom{2n-s}{s+a}^{\frac12} \binom{s+b}{b}^{-\frac12} \left(\frac{1-x}{2}\right)^{\frac{a}2} \left(\frac{1+x}{2}\right)^{\frac{b}2} P_s^{a,b}(x),
  \end{equation*}
  where $a=|k-l|$, $b=|k+l|$, $s=n-\max\{|k|,|l|\}$, and $\nu = \charFunc{k>l} \cdot(k+l)$.
  The function $P_s^{a,b}$ denotes the Jacobi polynomial of degree $s$ and parameters $a,b$, see~\cite{Szegoe1975}.
\end{definition}
Note that $\Wignerd{n}{k}{l}(\cdot)$ is a polynomial of degree $n$ if $k+l$ is even and $\sqrt{1-x^2}$ times a polynomial of degree $n-1$ otherwise.

Without the factor $\sqrt{2n+1}$, the Wigner-D functions correspond to the matrix elements of the irreducible unitary representations of $\SO3$, see~\cite{Vilenkin2012}, satisfying the representation property
\begin{equation}\label{eq:RepresentationProperty}
  \WignerD{n}{k}{l}(\mat R \mat Q) = \frac{1}{\sqrt{2n+1}} \sum_{j=-n}^n \WignerD{n}{k}{j}(\mat R) \, \WignerD{n}{j}{l}(\mat Q).
\end{equation}

The linear span of all Wigner-D functions with fixed harmonic degree $n\in\IN$ forms the harmonic subspace $\operatorname{Harm}_n(\SO3)$.
By the Peter-Weyl theorem, the collection of all Wigner-D functions constitutes a complete orthonormal basis of $\L2SO3$. Hence, every $f\in\L2SO3$ admits the unique harmonic expansion
\begin{equation*}
  f(\mat R) = \sum_{n=0}^{\infty}\sum_{k,l=-n}^n \fhat{n}{k}{l} \, \WignerD{n}{k}{l}(\mat R)
\end{equation*}
where $\fhat{n}{k}{l}  = \scp[\big]{f}{\WignerD{n}{k}{l}}_{\L2SO3}$ are the harmonic coefficients of $f$.
The space of $N$-band-limited functions is defined as
\begin{equation*}
  \bandlimFunc{N} = \bigoplus_{n=0}^{N} \operatorname{Harm}_n(\SO3)
\end{equation*}
which has dimension $\frac13(N+1)(2N+1)(2N+3)$ and corresponding index set
\[ \J{N} = \set[\big]{(n,k,l)}{ n=0,\dots,N \text{ and } k,l=-n,\dots,n}. \]


\section{The Double Fourier Sphere Method}\label{sec:DFS}

The classical Double Fourier Sphere (DFS) method \cite{Merilees1973,Orszag1974,Boyd1978,Yee1980} is based on the idea to transform a given function $f\colon\S{2}\to\IC$ onto the 2-dimensional torus $\IT^{2}$ by chaining it with the coordinate transform
\[ \phi_{\S2}\colon\IT^{2}\to\S{2}, \quad (\lambda,\theta)\mapsto(\cos\lambda\sin\theta,\sin\lambda\sin\theta,\cos\theta) \]
which double covers the sphere $\S2$, see~\cite{Mildenberger2022}. The transformed function $g(\lambda,\theta) = f(\phi_{\S2}(\lambda,\theta))$ can be expanded into a bivariate Fourier series, which enables quick Fourier methods.
However, because of double coverage, not every Fourier series on $\IT^{2}$ corresponds to a well-defined function on $\S2$.
Furthermore, the sphere's intrinsic curvature is lost in the mapping process, resulting in unavoidable distortions.
Due to these limitations, the approach of directly approximating spherical functions using Fourier series has not been established.

\subsection{The DFS Method on \texorpdfstring{$\SO3$}{SO(3)}}

In 2023 Mildenberger and Quellmalz \cite{Mildenberger2023} generalized the DFS method for a certain set of manifolds, in particular the rotation group $\SO3$.
The authors demonstrated that the Euler angle parameterization $\phi_{\SO3}\colon\IT^{3}\to\SO3$
enables us to represent rotational functions $f\colon\SO3\to\IC$ using functions on the torus $g\colon\IT^{3}\to\IC$ with $g = f\circ \phi_{\SO3}$, i.e.
\begin{align*}
\begin{tikzpicture}
\node (N1) at (0,0) {$\SO3$};
\node (N2) at (2.5,0) {$\IT^{3}$};
\node (N3) at (2.5,-2) {$\IC$};
\path[line width=0.8pt,->]
  (N2) edge node[above]{$\phi_{\SO3}$} (N1)
  (N2) edge node[right]{$g$} (N3)
  (N1) edge node[above]{} (N3);
\node at (1,-1.2) {$f$};
\end{tikzpicture}
\end{align*}
see~\cite[Sec. 6.5]{Mildenberger2023}.
More specifically, we have the following definition.
\begin{definition}\label{def:DFS-TrafoW}
  The DFS operator is a linear operator that maps rotational functions to their DFS transform by
  \[ \bm W \colon \C{}(\SO3) \to \C{}(\IT^{3}) \subset \Lp{2}(\IT^{3}) \]
  with $f\mapsto g=f\circ\phi_{\SO3}$.
\end{definition}

\begin{addedbox}
  In~\cite{Mildenberger2023}, the authors showed that Hölder regularity is preserved under the DFS operator. Furthermore, they analyzed the smoothness requirements on $f$ within the framework of Hölder spaces in order to obtain uniform convergence of the Fourier series of $\bm W(f)$.
  In contrast, we are more interested in characterizing how much Sobolev regularity of $f$ is necessary to ensure that $\bm W$ still maps it into $\Lp{2}(\IT^3)$.
  We will investigate the Sobolev regularity loss in detail in \Cref{sec:DFSSobolevSO3}.
\end{addedbox}
\deleted{In this context, we are ultimately interested in characterizing how much regularity of $f$ is necessary to ensure that $\bm W$ still maps it into $\Lp{2}(\IT^3)$.
In LEMMA 80 we provide an example showing that there exist functions in $\Sobolev{s}$  with $s<\tfrac12$ whose DFS transform does not belong to $\Lp{2}(\IT^3)$. This demonstrates that some regularity is lost under the DFS operator and, in particular, that it is not sufficient for $f$ to belong merely to $\Lp{2}(\SO3)$ or even to $\Sobolev{s}$ for $s<\frac12$.
In SECTION 3.30, we will show that, in fact, assuming $s>\frac34$ is sufficient to guarantee that $\bm W(f)\in \Lp{2}(\IT^3)$.}

\deleted{Furthermore, in [230,THM. 5.3] the authors analyzed the smoothness requirements on $f$ within the framework of Hölder spaces in order to obtain uniform convergence of the Fourier series of $\bm W(f)$.}

Due to the double coverage of $\SO3$, the parametrization satisfies
\[ {\phi_{\SO3}(\alpha,\beta,\gamma) = \phi_{\SO3}(\alpha+\pi,-\beta,\gamma+\pi)}, \]
and thus the transformed function exhibits a so-called block-mirror centrosymmetry (BMC) property, i.e.,
\[ g(\alpha,\beta,\gamma) = g(\alpha+\pi,-\beta,\gamma+\pi) \]
which directly transfers to its Fourier coefficients.
\begin{lemma}\label{lem:SymGhatDFSMethod}
  For $(k,j,l)\in\IZ^{3}$, the Fourier coefficients of the BMC function $g\in\Lp{2}(\IT^{3})$ satisfy
  $\ghat{k}{j}{l} = (-1)^{k+l}\,\ghat{k}{-j}{l}$.
\end{lemma}
\begin{proof}
  It holds
  \begin{align*}
    \ghat{k}{j}{l}
    &= \int_{\IT^{3}} g(\alpha,\beta,\gamma)\,\e^{-\i\,(k,j,l)\cdot(\alpha,\beta,\gamma)^{\top}} \d{(\alpha,\beta,\gamma)}\\
    &= \int_{\IT^{3}} g(\alpha+\pi,-\beta,\gamma+\pi)\,\e^{-\i\,(k,j,l)\cdot(\alpha+\pi,-\beta,\gamma+\pi)^{\top}} \d{(\alpha,\beta,\gamma)}\\
    &= \e^{-\i\,\pi\,(k+l)} \cdot \int_{\IT^{3}} g(\alpha,\beta,\gamma)\,\e^{-\i\,(k,-j,l)\cdot(\alpha,\beta,\gamma)^{\top}} \d{(\alpha,\beta,\gamma)}
    = (-1)^{k+l} \, \ghat{k}{-j}{l}.
    \end{align*}
\end{proof}
We shall make use of this symmetry in \Cref{sec:NFSOFT} to accelerate our algorithms.

In the following chapters we analyze the DFS operator $\bm W$ in Fourier space.

\subsection{The DFS Method for Band-Limited Functions on \texorpdfstring{$\SO3$}{SO(3)}}\label{sec:DFSBandlimitedSO3}

Let $f\in\bandlimFunc{N}$ be band-limited with bandwidth $N\in\IN$, i.e. it has a unique harmonic representation
\begin{align}\label{eq:BandLimitedHarmonicSeries}
  f(\mat R) = \sum_{n=0}^{N}\sum_{k,l=-n}^{n} \fhat{n}{k}{l}\,D_{n}^{k,l}(\mat R).
\end{align}

In the next theorem, we study the Fourier representation of the band-limited DFS operator $\bm W_{N}$, which maps $f$ to an element of
\[ \changedmath{\trigonometricPolynomials{N}} = \linspan \set*{ \e^{\i \, \bm m \, \cdot \bm x^{\top}} }{\bm m\in\IZ^{3}, \norm{\bm m}_{\infty}\leq N}, \qquad \bm x\in\IT^{3},\]
the space of all trigonometric polynomials of degree at most $N$.
\begin{theorem}
  Let $\mathcal F_{\IT^{3}}\colon\ell_{2}(\IZ^{3})\to\Lp{2}(\IT^{3})$ and $\mathcal F_{\SO3}\colon \ell_{2}(\J{\infty})\to\L2SO3$ denote the Fourier transforms on $\IT^{3}$ and $\SO3$, respectively.
  Then, the band-limited DFS operator $\bm W_{N} \colon \bandlimFunc{N} \to \changedmath{\trigonometricPolynomials{N}}$ reads as
  \[ \bm W_{N} = \mathcal F_{\mathbb T^{3}} \, \coeffTrafo{N} \, \mathcal F^{-1}_{\SO3}, \]
  where the linear operator $\coeffTrafo{N}$ is defined by its matrix vector product
  \begin{align}\label{eq:WignerCoeffTrafo}
    \big( \coeffTrafo{N} \, \fhat \big)_{k,j,l}  \changedmath{\coloneqq}  \i^{k-l} \hspace{-15pt}  \sum_{n=\max\{|k|,|l|,|j|\}}^N \hspace{-15pt} \sqrt{2n+1} \, \Wignerd{n}{j}{k}(0) \, \Wignerd{n}{j}{l}(0) \, \fhat{n}{-k}{-l}
  \end{align}
  for all $k,j,l = -N,\dots,N$.
\end{theorem}
\begin{proof}
  Insertion of the specific choice of rotations
  $\mat R_y(\beta) = \mat R(\tfrac{\pi}2,\tfrac{\pi}2,\pi)\,\mat R(\beta,\tfrac{\pi}2,\tfrac{\pi}2)$
  in the representation property from \cref{eq:RepresentationProperty} and using the symmetry property
  $\Wignerd{n}{k}{j}(0) = (-1)^{k+j} \, \Wignerd{n}{j}{k}(0) = (-1)^{n+j} \, \Wignerd{n}{-j}{k}(0)$ (see~\cite{Varshalovich1988}) yields
  \begin{align}\label{eq:PolynomTrafo}
    \Wignerd{n}{k}{l}(\cos\beta) =\i^{l-k} \,  \sum_{j=-n}^n \Wignerd{n}{j}{k}(0) \, \Wignerd{n}{j}{l}(0) \,\e^{\i\,j\,\beta}
  \end{align}
  as Fourier expansion of the Wigner-d functions, see~\cite{Bunge1982}.

  Now, let $f\in\bandlimFunc{N}$ with harmonic coefficient vector $\fhat \coloneqq \mathcal F^{-1}_{\SO3}(f)$ be given.
  Then it yields
  \[ \bm W(f) = f(\mat R(\alpha,\beta,\gamma)) = \sum_{n=0}^{N}\sum_{k,l=-n}^{n} \fhat{n}{k}{l}\,D_{n}^{k,l}(\mat R(\alpha,\beta,\gamma)). \]
  By substituting the product approach of the Wigner-D functions and changing the order of summation we obtain
  \begin{align*}
    \bm W(f) =  \sum_{k,l=-N}^N \e^{\i\,(k,l)\cdot(\alpha,\gamma)^{\top}} \,\sum_{n=\max\set{\abs{k},\abs{l}}{}}^N \sqrt{2n+1}\,\fhat{n}{-k}{-l}\,\Wignerd{n}{-k}{-l}(\cos\beta)
  \end{align*}
  Now the Fourier expansion \eqref{eq:PolynomTrafo} from the first part of the proof implies
  \begin{align*}
    \bm W(f) = \sum_{k,j,l=-N}^{N} \braces*{\,\i^{k-l} \hspace{-20pt}  \sum_{n=\max\{|k|,|l|,|j|\}}^N \hspace{-20pt} \sqrt{2n+1}\, \fhat{n}{-k}{-l}\, \Wignerd{n}{j}{k}(0) \, \Wignerd{n}{j}{l}(0) } \, \e^{\i\,(k,j,l)\cdot(\alpha,\beta,\gamma)^{\top}}.
  \end{align*}
\end{proof}

Note that the linear operator $\coeffTrafo{N}$ represents the DFS operator in the frequency domain. We will call it Wigner transform, since it transforms harmonic (Wigner-D) coefficients into Fourier coefficients.
The transforms of the previous theorem can be visualized as follows:
\begin{align*}
  \begin{tikzpicture}
    \node (N1) at (0,0) {$\bandlimFunc{N}$};
    \node (N2) at (3,0) {\changed{$\trigonometricPolynomials{N}$}};
    \node (N3) at (0,-2) {$\IC^{\abs{\J{N}}}$};
    \node (N4) at (3,-2) {$\IC^{(2N+1)^{3}}$};
    \path[line width=0.8pt,->]
    (N1) edge node[above]{$\bm W_{N}$} (N2)
    (N3) edge node[above]{$\coeffTrafo{N}$} (N4)
    (N1) edge node[left]{$\mathcal F^{-1}_{\SO3}$} (N3)
    (N4) edge node[right]{$\mathcal F_{\IT^{3}}$} (N2);
  \end{tikzpicture}
\end{align*}

\begin{remark}\label{rem:WignerTrafoPotts}
  Transforming a function on $\SO3$ into a function on $\IT^{3}$ is the key idea behind all~\cite{Potts2009,Kostelec2008} fast algorithms for estimating the harmonic series~\eqref{eq:BandLimitedHarmonicSeries}.
  In \cite{Potts2009}, the authors implemented the Wigner transform $\coeffTrafo{N}$ via a fast polynomial transform.
  \added{In particular, they perform a change of polynomial basis to the Chebyshev polynomials $T_{n}(\cos\beta)=\cos(n\beta)$, yielding}
  \begin{equation*}
    \sum_{n=\max\set{\abs{k},\abs{l}}{}}^N\sqrt{2n+1}\,\fhat{n}{k}{l}\,\Wignerd{n}{k}{l}(\cos\beta) =
    \begin{cases}
      \ds\sum_{n=0}^{N} \harmCoeff{h}{n}{k}{l} \, T_{n}(\cos\beta), & \text{if } k+l \text{ even},\\
      \ds\sin(\beta)\cdot\sum_{n=0}^{N-1} \harmCoeff{h}{n}{k}{l} \, T_{n}(\cos\beta), & \text{if } k+l \text{ odd},
    \end{cases}
  \end{equation*}
  for any $k,l=-N,\dots,N$, \added{where $\harmCoeff{h}{n}{k}{l}$ are the Chebyshev coefficients.}
  A second change of basis easily transforms the cosine series into a suitable Fourier series.

  Consequently, this scheme provides a concrete realization of the Wigner transform $\coeffTrafo{N}$, and thus represents another realization of the band-limited DFS operator $\bm W_{N}$ in Fourier space.
\end{remark}

In \Cref{sec:NFSOFT}, we will discuss a direct implementation of $\coeffTrafo{N}$ based on \Cref{eq:WignerCoeffTrafo} and compare it with the fast approach proposed in \cite{Potts2009}.


\subsection{The DFS-Method for Non-Band-Limited Functions on \texorpdfstring{$\SO3$}{SO(3)}}\label{sec:DFSSobolevSO3}

In the previous section, we analyzed the DFS operator in the frequency domain for band-limited functions on $\SO3$.
We now extend this framework to non-band-limited functions in $\L2SO3$ that possess sufficient regularity.
Since our analysis is carried out in Fourier space, we focus in particular on the decay behavior of the harmonic coefficients, which is most naturally characterized using Sobolev spaces.

In what follows, we show that the Fourier representation of the non-band-limited DFS operator $\bm{W}$ coincides with the Wigner transform $\coeffTrafo{N}$ as $N$ grows to infinity. In this context, we examine the Sobolev regularity required to ensure that the range of $\bm W$ is contained in $\Lp{2}(\IT^{3})$.

\begin{definition}
  Let $s \geq 0$.
  The Sobolev space of harmonic coefficients $\SobolevInd{s}$ is defined as the set of all vectors $\fhat = (\fhat{n}{k}{l})_{(n,k,l)\in\J{\infty}}$ for which $\norm{\fhat}_{\SobolevInd{s}} < \infty$,
  where the norm is induced by the inner product
  \[ \scp{\fhat}{\ghat}_{\SobolevInd{s}} \coloneqq \sum_{(n,k,l)\in\J{\infty}} (1+n(n+1))^{s} \, \fhat{n}{k}{l} \, \conj{\harmCoeff{g}{n}{k}{l}} ,\qquad \text{for } \fhat,\ghat\in\SobolevInd{s}.\]
  Accordingly, the Sobolev space $\Sobolev{s}$ is defined by
  \[ \Sobolev{s} \coloneqq \set*{ f \in \L2SO3 }{ \fhat \in \SobolevInd{s} }.\]
\end{definition}

\begin{remark}
  By the Sobolev embedding theorem, we have $\Sobolev{s} \hookrightarrow \C{}(\SO3)$ for $s>\frac32$.
  Consequently, the DFS operator is well defined on $\Sobolev{s}$ and induces the continuous embedding
  \[ \bm W \colon \Sobolev{s} \to \Lp{2}(\IT^{3}), \qquad s>\tfrac32.\]
\end{remark}

\begin{addedbox}
This raises the question: How much Sobolev regularity is necessary to ensure that $\bm W$ still maps into $\Lp{2}(\IT^{3})$?

In the rest of this section, we will show that the smallest possible Sobolev index $s$ lies in the interval $\left[\frac12,\frac34+\varepsilon\right]$ for every $\varepsilon>0$.
The next example establishes the lower bound $s\geq\frac12$, while the upper bound $s>\frac34$ is proved at the end of this section.
\end{addedbox}

\deleted{In the remainder of this section, we refine the Sobolev regularity condition introduced above and examine the smoothness required for the DFS operator to map rotational functions to functions on $\IT^3$.
The following lemma illustrates the loss of regularity that may occur under this mapping.}

\begin{lemma}\label{lem:CounterExampleSobolev1/2}
  Let
  \begin{align*}
  f \colon \SO3 \to \IC, ~ f(\mat R(\alpha,\beta,\gamma)) =
    \begin{cases}
      \frac1{\sqrt{\sin \beta}}, & \text{if } \beta\notin \set{0,\pi}{}, \\
      0, & \text{otherwise}.
    \end{cases}
  \end{align*}
  Then $f\in\Sobolev{s}$ for $s<\frac12$ and $f\circ\phi_{\SO3} \notin \Lp{2}(\IT^{3})$.
\end{lemma}
\begin{proof}
  It is immediate that $f\circ\phi_{\SO3} \notin \Lp{2}(\IT^{3})$, since
  \begin{align*}
    \int_{\IT^{3}} \abs{f(\mat R(\alpha,\beta,\gamma))}^{2} \d{\mat R}
    &= \frac1{2\pi} \int_{0}^{2\pi} \frac1{\abs{\sin\beta}} \d{\beta}
    = \frac1{\pi} \int_{0}^{\pi} \frac1{\sin\beta} \d{\beta}
    = \frac2{\pi} \lim_{a\to0+} \int_{a}^{\frac\pi2} \frac1{\sin\beta} \d{\beta} \\
    &= \frac2{\pi} \lim_{a\to0+} \ln\abs*{\tan\frac\beta2} \Big|_{a}^{\frac\pi2}
    = - \lim_{a\to0+} \frac2\pi \ln\braces{\tan\frac{a}2} = \infty.
  \end{align*}
  It is also straightforward to see that $f\in\L2SO3$.
  It remains to show that the vector of harmonic coefficients $\fhat$ belongs to $\SobolevInd{s}$.
  By definition of the harmonic coefficients, we have
  \begin{align*}
    \fhat{n}{k}{l} = \delta_{k,0}\cdot\delta_{l,0}\cdot\frac{\sqrt{2n+1}}2 \cdot \int_{0}^{\pi} \sqrt{\sin\beta} \, \Wignerd{n}{k}{l}(\cos\beta) \d{\beta}.
  \end{align*}
  Hence, $\fhat{n}{k}{l}=0$ if $k\neq0$ or $l\neq0$.
  Using $\Wignerd{n}{0}{0}(\cos\beta) = P_{n}(\cos\beta)$, where $P_{n}$ denotes the Legendre polynomial of degree $n$, and making the substitution $t=\cos\beta$ with ${\d{t}=-\sin\beta\d{\beta}}$, it follows that
  \begin{align*}
    \fhat{n}{0}{0} = \frac{\sqrt{2n+1}}2 \cdot \underbrace{ \int_{-1}^{1} P_{n}(t)\cdot \frac1{\sqrt[4]{1-t^{2}}} \d{t} }_{\eqqcolon I}.
  \end{align*}
  By \cite[(7.132.1)]{Gradshteyn1980}, for $n>0$ we obtain
  \begin{align*}
    I = \frac{\pi\,\Gamma(\frac34)^{2}} {\Gamma(\frac{n}2+\frac54) \, \Gamma(\frac34-\frac{n}2) \, \Gamma(\frac{n}2+1) \, \Gamma(\frac12-\frac{n}2)}
  \end{align*}
  Using $\Gamma(\frac12-\frac{n}2) = \frac{(-1)^{n/2}\,2^{n}\,\sqrt{\pi}\,(\frac{n}2)!}{n!}$ and $\Gamma(\frac{n}2+\frac54)\,\Gamma(\frac34-\frac{n}2) = (-1)^{n/2}\,\sqrt{2}\,\pi\,(\frac{n}2+\frac14)$ Stirling's formula yields
  \begin{align*}
    I = \frac{\Gamma(\frac34)^{2}\sqrt{2}}{\sqrt{\pi}} \cdot \frac1{(n+\frac12)}\cdot\frac1{2^{n}}\,\binom{n}{\frac{n}2} = \frac{2\,\Gamma(\frac34)^{2}\, \xi_{n}}{\pi} \cdot \frac1{\sqrt{n}\,(n+\frac12)}
  \end{align*}
  for some $\xi_{n}\in(\frac78,1)$.
  Hence,
  \[\norm{f}^{2}_{\Sobolev{s}} = \sum_{n=0}^{\infty} (1+n\,(n+1))^{s} \, \abs*{\fhat{n}{0}{0}}^{2} < \frac65 + \frac12 \sum_{n=1}^{\infty} \frac{(1+n\,(n+1))^{s}}{n\,(n+\frac12)}, \]
  which converges for $s<\frac12$.
\end{proof}

Now we prove a lemma on the asymptotic behavior of series involving Wigner-d functions.
This result plays a crucial role in analyzing the Wigner transform $\coeffTrafo{N}$ as
$N$ grows to infinity.
\begin{lemma}\label{lem:seriesEstimate}
  Let $k,l\in\IZ$ and \changed{$s>\frac34$}. Then there exists a constant $C>0$, independent of $k$ and $l$, such that
  \begin{align}\label{eq:seriesEstimate}
    \sum_{n=\max{\set{\abs{k},\abs{l}}{}}}^{\infty} \frac1{(2n+1)^{2s-1}} \sum_{j=-n}^{n} \abs{\Wignerd{n}{j}{k}(0)\Wignerd{n}{j}{l}(0)}^{2} < C
  \end{align}
  for all $k,l\in\IZ$.
\end{lemma}
\begin{addedbox}
\begin{proof}
  Using the Fourier expansion of $\Wignerd{n}{k}{l}\circ\cos$, given in \Cref{eq:PolynomTrafo}, the Fourier coefficients satisfy
  \begin{equation*}
    \widehat{\bigl(\Wignerd{n}{k}{l}\circ\cos\bigr)}_{j} =
    \begin{cases}
      \i^{l-k} \Wignerd{n}{j}{k}(0)\Wignerd{n}{j}{l}(0), & \text{if } \abs{j}\leq n,\\
      0, & \text{otherwise},
    \end{cases}
  \end{equation*}
  with respect to the $\Lp{2}(\IT)$ inner product.
  Hence, by Parseval's identity,
  \[ I_{n} \coloneqq \sum_{j=-n}^{n} \abs{\Wignerd{n}{j}{k}(0)\Wignerd{n}{j}{l}(0)}^{2} = \frac1{2\pi}\int_{-\pi}^{\pi}\abs*{\Wignerd{n}{k}{l}(\cos\beta)}^{2}\d\beta. \]
  Since $\Wignerd{n}{k}{l}(\cos(-\beta))=(-1)^{k-l}\Wignerd{n}{k}{l}(\cos\beta)$ (\cite[Sec.~4.4]{Varshalovich1988}), the integrand is even. Hence, after substitution, we  obtain
  \[ I_{n} = \frac1\pi \int_{-1}^{1} \frac{\abs{\Wignerd{n}{k}{l}(x)}^{2}}{\sqrt{1-x^{2}}} \d{x}. \]
  We decompose the integration domain into three regions
  \[ R_{n}^{1}=(-1,-1+\tfrac1n), \quad R_{n}^{2}=(-1+\tfrac1n,1-\tfrac1n), \quad R_{n}^{3}=(1-\tfrac1n,1) \]
  and denote the corresponding integrals by $I_{n}=I_{n}^{1}+I_{n}^{2}+I_{n}^{3}$.

  Throughout, we assume $n>1$.

  \medskip
  \noindent
  \underline{Boundary Regions:} For $I_{n}^{1}$ and $I_{n}^{3}$ we use the bound $\abs{\Wignerd{n}{k}{l}(x)}\leq1$, which holds since the Wigner-d functions are matrix elements of the unitary representation $\WignerD{n}{k}{l}(\mat R_{\vec y}(\beta))$. Hence
  \[ I_{n}^{1} \leq \int_{-1}^{-1+\frac1n} \frac1{\sqrt{1-x^{2}}}\d{x} \leq \int_{-1}^{-1+\frac1n} \frac1{\sqrt{1+x}}\d{x} = \frac2{\sqrt{n}}. \]
  An analogous bound holds for $I_{n}^{3}$.

  \medskip
  \noindent
  \underline{Bulk Region:} For $I_{n}^{2}$ we use the uniform upper bound from~\cite{Haagerup2013},
  \[ (1-x^{2})^{\frac14} \, \abs{\Wignerd{n}{k}{l}(x)} \leq 12 \,(2n+1)^{-\frac14}, \]
 which implies
  \[  I_{n}^{2} \leq \frac{144}{\sqrt{2n+1}} \, \int_{-1+\frac1n}^{1-\frac1n} \frac1{1-x^{2}} \d{x} = \frac{144\,\ln(2n-1)}{\sqrt{2n+1}}. \]

  Combining the three estimates, we conclude that $I_{n}\leq\frac{151\,\ln(2n+1)}{\sqrt{2n+1}}$ for $n>1$.
  Assuming $s>\frac34$ we obtain
  \begin{align*}
    \sum_{n=\max{\set{\abs{k},\abs{l}}{}}}^{\infty} \frac1{(2n+1)^{2s-1}} \underbrace{\sum_{j=-n}^{n} \abs{\Wignerd{n}{j}{k}(0)\Wignerd{n}{j}{l}(0)}^{2}}_{I_{n}}
    &\leq \sum_{n=\max{\set{\abs{k},\abs{l}}{}}}^{\infty} \frac{151 \, \ln(2n+1)}{(2n+1)^{2s-\frac12}}.
  \end{align*}
\end{proof}
\end{addedbox}
\begin{changedbox}
Note that in~\cite{Haagerup2013} it is stated that this uniform upper bound is optimal as $\abs{k},\abs{l}$ tend to infinity.
If $k$ and $l$ are fixed, one obtains an upper bound of order $\mathcal{O}(n^{-\frac12})$ using the asymptotics of the Wigner-D functions, cf.~\cite{Varshalovich1988}. However, this bound is not uniform in $k$ and $l$.
\end{changedbox}

We use \Cref{lem:seriesEstimate} to extend the Wigner transform $\coeffTrafo{N}$ to the Sobolev space $\SobolevInd{s}$ in the non-band-limited setting.
\begin{theorem}\label{thm:WOperatorBounded}
  Let \changed{$s>\frac34$}.
  The Wigner transform $\coeffTrafo{N}$ from \Cref{thm:WOperatorBounded} extends to a bounded linear operator $\coeffTrafo\colon\SobolevInd{s}\to\ell_{2}(\IZ^{3})$, obtained as the pointwise limit of
  $\coeffTrafo{N}$ acting on the projections of $\fhat$ onto the first $\abs{\J{N}}$ modes.

  The operator is explicitly given by
  \[ \big( \coeffTrafo \, \fhat \big)_{k,j,l}   =  \i^{k-l} \hspace{-20pt}  \sum_{n=\max\{|k|,|l|,|j|\}}^{\infty} \hspace{-20pt} \sqrt{2n+1} \, \Wignerd{n}{j}{k}(0) \, \Wignerd{n}{j}{l}(0) \, \changedmath{\fhat{n}{-k}{-l}} \]
  for $(k,j,l)\in \IZ^{3}$.
\end{theorem}

\begin{proof}
  First we show, that $\coeffTrafo{N}$ converges pointwise to $\coeffTrafo$ in $\ell_{2}(\IZ^{3})$.

  By Cauchy-Schwarz inequality, we obtain
  \begin{align*}
    &\norm*{ \coeffTrafo \fhat - \coeffTrafo_{N} \fhat  }_{\ell_{2}(\IZ^{3})}^{2} = \sum_{(k,j,l)\in\IZ^{3}} \abs*{ \sum_{n=\max\{|k|,|j|,|l|,N\}}^{\infty} \hspace{-20pt} \sqrt{2n+1}\,\Wignerd{n}{j}{k}(0)\,\Wignerd{n}{j}{l}(0)\,\changedmath{\fhat{n}{-k}{-l}} }^{2} \\
    &\leq \sum_{(k,j,l)\in\IZ^{3}} \braces*{  \sum_{n=\max\{|k|,|j|,|l|,N\}}^{\infty} \hspace{-20pt} \tfrac{2n+1}{(1+n\,(n+1))^{s}} \, \abs*{\Wignerd{n}{j}{k}(0) \, \Wignerd{n}{j}{l}(0) }^{2} } \cdot \braces*{ \sum_{n=\max\{|k|,|j|,|l|,N\}}^{\infty} \hspace{-20pt} (1+n\,(n+1))^{s} \, \changedmath{\abs{\fhat{n}{-k}{-l}}^{2}} }\\
    &\leq \sum_{(k,l)\in\IZ^{2}} \braces*{ \sum_{j\in\IZ} \sum_{n=\max\{|k|,|j|,|l|\}}^{\infty}  \hspace{-20pt} \tfrac{2n+1}{(1+n\,(n+1))^{s}} \,\abs*{\Wignerd{n}{j}{k}(0)\,\Wignerd{n}{j}{l}(0)}^{2} } \cdot \braces*{ \sum_{n=\max\{|k|,|l|,N\}}^{\infty}  \hspace{-20pt} (1+n(n+1))^{s} \, \changedmath{\abs{\fhat{n}{-k}{-l}}^{2}} }.
  \end{align*}
  By \Cref{lem:seriesEstimate}, the term in the first braces is uniformly bounded by a constant $C$ independent of $k$ and $l$.
  Changing the order of summation yields
  \begin{align*}
    \norm*{\coeffTrafo\fhat - \coeffTrafo_{N} \fhat}^{2}_{\ell_{2}(\IZ^{3})} \leq C \cdot \sum_{n=N}^{\infty} (1+n\,(n+1))^{s} \,\sum_{k,l=-n}^{n}  \abs{\fhat{n}{k}{l}}^{2}.
  \end{align*}
  For any fixed $\fhat\in\SobolevInd{s}$, the right-hand side tends to zero as $N\to\infty$. Hence, $\coeffTrafo{N}$ converges pointwise to $\coeffTrafo$.

  Analogously, one obtains
  \[\norm{\coeffTrafo\fhat}^{2}_{\ell_{2}(\IZ^{3})} \leq C \cdot \norm{\fhat}_{\SobolevInd{s}}^{2},\]
  that is, $\coeffTrafo$ is a bounded operator from $\SobolevInd{s}$ to $\ell_{2}(\IZ^{3})$.
\end{proof}

This theorem immediately yields the Fourier-space representation of the DFS operator.
\begin{corollary}
  Let \changed{$s>\frac34$} and let $\mathcal F_{\IT^{3}}\colon\ell_{2}(\IZ^{3})\to\Lp{2}(\IT^{3})$ and $\mathcal F_{\SO3}\colon \ell_{2}(\J{\infty})\to\L2SO3$ be the Fourier transforms on $\IT^{3}$ and $\SO3$, respectively.
  Then the DFS operator from \Cref{def:DFS-TrafoW}, \changed{reads as}
  \begin{addedbox}
    \[ \bm W\colon\Sobolev{s}\to\Lp{2}(\IT^{3}), \qquad \bm W = \mathcal F_{\mathbb T^{3}} \, \coeffTrafo \, \mathcal F^{-1}_{\SO3}, \]
  \end{addedbox}
  where the linear operator $\coeffTrafo$ is defined in \Cref{thm:WOperatorBounded},
  i.e.
  \begin{align*}
  \begin{tikzpicture}
    \node (N1) at (0,0) {$\Sobolev{s}$};
    \node (N2) at (3,0) {$\Lp{2}(\IT^{3})$};
    \node (N3) at (0,-2) {$\SobolevInd{s}$};
    \node (N4) at (3,-2) {$\ell_{2}(\IZ^{3})$};
    \path[line width=0.8pt,->]
    (N1) edge node[above]{$\bm W$} (N2)
    (N3) edge node[above]{$\coeffTrafo$} (N4)
    (N1) edge node[left]{$\mathcal F^{-1}_{\SO3}$} (N3)
    (N4) edge node[right]{$\mathcal F_{\IT^{3}}$} (N2);
  \end{tikzpicture}
  \end{align*}
\end{corollary}
\begin{proof}
  Let $f\in \Sobolev{s}$. Then by definition of the Sobolev space $\Sobolev{s}$ the harmonic coefficient vector $\fhat =\mathcal F_{\SO3}^{-1} f$ is in $\SobolevInd{s}$.
  By \Cref{thm:WOperatorBounded} the operator $\coeffTrafo$ maps the harmonic coefficient vector $\fhat$ to the Fourier coefficient vector $\ghat\in\ell_{2}(\IZ^{3})$ where
  \[ \sum_{(n,k,l)\in\J{\infty}} \fhat{n}{k}{l}\,\WignerD{n}{k}{l}(\mat R(\alpha,\beta,\gamma)) = \sum_{(k,j,l)\in\IZ^{3}} \ghat{k}{j}{l}\,\e^{-\i\,(k,j,l)\cdot(\alpha,\beta,\gamma)^{\top}}. \]
  The Fourier transform $\mathcal F_{\IT^{3}}$ of $\ghat$ yields the corresponding Fourier series $g = \mathcal F_{\IT^{3}}\ghat \in\Lp{2}(\IT^{3})$, with
  \[g(\vec x) = \sum_{(k,j,l)\in\IZ^{3}} \ghat{k}{j}{l}\,\e^{-1\i(k,j,l)\cdot\vec x^{\top}}.\]
\end{proof}

\begin{addedbox}
Similarly, we obtain the following corollary, which quantifies the Sobolev regularity loss under the DFS operator.
\begin{corollary}\label{Cor:SobolevRegularityLoss}
  Let $t\geq0$. The DFS operator $\bm W$ restricts to a bounded linear operator
  \[ \bm W\colon\Sobolev{t+s}\to\Sobolev{t}{\IT^{3}} \]
  for all $s>\frac34$.
\end{corollary}
Note that this bound is not sharp.
\end{addedbox}


\section{Fast Algorithms for Harmonic Series on \texorpdfstring{$\SO3$}{SO(3)}}\label{sec:NFSOFT}

In the previous chapter, we studied the DFS operator $\bm W$, which maps rotational functions to functions on the torus.
Its representation in Fourier space led to the Wigner transform $\coeffTrafo{N}$, which converts harmonic series on $\SO3$ into Fourier series on $\IT^{3}$.
This, in turn, allows us to analyze fast algorithms on $\SO3$ by mapping the problem back to the torus and employing fast Fourier methods there.

In this chapter, we investigate efficient algorithms for the nonequispaced $\SO3$-Fourier transform (NSOFT), which enables the evaluation of band-limited harmonic series at arbitrary rotations, as well as its adjoint on suitable quadrature grids, leading to efficient inversion schemes.
Furthermore, we study how symmetry properties of functions on $\SO3$ can be exploited to accelerate computations.

\subsection{Factorization of the \texorpdfstring{$\SO3$}{SO(3)}-Fourier Transform}\label{sec:NSOFTOperators}

The $\SO3$-Fourier transform is a linear operator, that evaluates a band-limited harmonic series $f\in\bandlimFunc{N}$, as defined in \cref{eq:BandLimitedHarmonicSeries}, at arbitrary rotations $\mathcal R = \set{\mat R_{1},\dots,\mat R_{M}}{}\subset\SO3$.
It can be expressed as the matrix-vector product
\[ \bm f = \bm D_{\mathcal R,N} \, \fhat, \]
where
\[ \fhat = \big(\fhat{n}{k}{l}\big)_{(n,k,l)\in\J{N}} \in \IC^{\abs{\J{N}}} \]
denotes the vector of harmonic coefficients,
\[ \bm f = \big(f(\mat R_m)\big)_{m=1}^M \in \IC^M \]
is the vector of function values, and
\[ \bm D_{\mathcal R,N} = \big( \WignerD{n}{k}{l}(\mat R_m) \big)_{m\in\set{1,\dots,M}{} ,\, (n,k,l)\in\J{N}} \in \IC^{M \times \abs{\J{N}}} \]
is the nonequispaced $\SO3$-Fourier matrix (Wigner-D matrix).

Using the results of the previous chapter, we obtain the factorization
\[ \bm D_{\mathcal R,N} = \FourierTrafo{\mathcal R,N} \, \coeffTrafo{N}, \]
where $\coeffTrafo{N}$ denotes the Wigner transform, see \Cref{def:DFS-TrafoW} and
\begin{equation}\label{eq:FourierMatrix}
  \FourierTrafo{\mathcal R,N} = \left( \e^{-\i\,(k,j,l)\cdot(\alpha_{m},\beta_{m},\gamma_{m})^{\top}} \right)_{m\in\{1,\dots,M\};~(k,j,l)\in\{-N,\dots,N\}^{3}}
\end{equation}
is the Fourier transform, which can be computed efficiently using the nonequispaced fast Fourier transform (NFFT), see~\cite{Potts2001}.

The adjoint $\SO3$-Fourier transform reads as
\[ \bm D^{H}_{\mathcal{R},N} = \coeffTrafo{N}^{H}\,\FourierTrafo{\mathcal{R},N}^{H}, \]
where the adjoint Wigner transform is defined in the following lemma.

\begin{lemma}\label{lem:AdjointCoeffTrafo}
  Let $N\in\IN$ and $\ghat =\big( \ghat{k}{j}{l} \big)_{k,j,l=-N}^N \in \IC^{(2N+1)^3}$ be given. Then we have
  \begin{align}\label{eq:AdjointCoeffTrafoRisbo}
    \big( \coeffTrafo{N}^H \, \ghat \big)_n^{k,l} = \sqrt{2n+1}\,\i^{l-k} \sum_{j=-n}^n \Wignerd{n}{j}{k}(0) \, \Wignerd{n}{j}{l}(0) \, \ghat{k}{j}{l}
  \end{align}
  for all triples $(n,k,l)\in\J{N}$.
\end{lemma}
\begin{proof}
  Let $\fhat\in\mathbb{C}^{|\mathcal{J}_N|}$. Since the Wigner-d functions are real-valued, it yields
  \begin{align*}
    \scp{\coeffTrafo{N}\,\fhat}{\ghat}_{2}
    &= \sum_{k,j,l=-N}^N \braces*{\, \i^{k-l} \hspace{-20pt}  \sum_{n=\max\{|k|,|j|,|l|\}}^N \hspace{-20pt} \sqrt{2n+1} \, \Wignerd{n}{j}{k}(0) \, \Wignerd{n}{j}{l}(0) \,\fhat{n}{k}{l} } \, \overline{\ghat{k}{j}{l}} \\
    &= \sum_{n=0}^N \sum_{k,l=-n}^n  \fhat{n}{k}{l} \, \overline{\braces*{ \sqrt{2n+1} \,\i^{l-k} \sum_{j=-n}^n \Wignerd{n}{j}{k}(0) \, \Wignerd{n}{j}{l}(0) \, \ghat{k}{j}{l} }}
      = \scp{\fhat}{\coeffTrafo{N}^H\,\ghat}_2.
  \end{align*}
\end{proof}

\subsection{Computation of the \texorpdfstring{$\SO3$}{SO(3)}-Fourier Coefficients}\label{sec:InverseNSOFT}

The adjoint $\SO3$-Fourier transform plays a crucial role in computing the harmonic coefficients
\[ \fhat{n}{k}{l} = \scp{f}{\WignerD{n}{k}{l}}_{\L2SO3}, \quad (n,k,l)\in\J{N}, \]
of a given $N$-band-limited function $f\in\bandlimFunc{N}$ via numerical integration.
Using an exact quadrature rule with nodes $\tilde{\mathcal{R}} = \set{\mat R_1,\dots,\mat R_{M}}{}$ and weights $(\omega_m)_{m=1}^{M}$, these integrals reduce to
\[ \fhat{n}{k}{l} = \sum_{m=1}^{M} \omega_m \, f(\mat R_m) \, \overline{\WignerD{n}{k}{l}(\mat R_m)},\]
which essentially is the adjoint NSOFT on a weighted vector of function values. In matrix-vector notation it reads as
\[ \fhat = \bm D^{H}_{\tilde{\mathcal{R}},N} \, \mathrm{diag}(\omega_{m}) \, \bm f \]
and therefore $\bm D^{H}_{\tilde{\mathcal{R}},N} \, \mathrm{diag}(\omega_{m})$ is the left-inverse of the Wigner transform.

To enable exact computation of the harmonic coefficients of $N$-band-limited functions, we adopt a multiplicative quadrature scheme with respect to the Euler angles, using Gaussian quadrature along the first and third angles and Clenshaw-Curtis quadrature along the second Euler angle $\beta$, see~\cite{Potts2009}.
This construction yields an equispaced rotation grid, allowing the Fourier matrix $\FourierTrafo{\mathcal{R},N}^{H}$, as part of the adjoint NSOFT, to be computed via an equispaced trivariate FFT, which is significantly faster than the NFFT.

Following \cite{Khalid2015}, a Gauss-Legendre quadrature can be used instead of Clenshaw-Curtis, requiring only half as many nodes along the second Euler angle $\beta$. Since these nodes are nonequispaced, the Fourier matrix can be computed via a univariate NFFT combined with a bivariate FFT.

For further results concerning quadrature formulas on $\SO3$, see~\cite{Graef2008,Graef2009,Graef2011}.


\subsection{Symmetry Properties on \texorpdfstring{$\SO3$}{SO(3)}}\label{sec:Symmetries}

In many applications, such as crystallography, functions on $\SO3$ are real-valued and exhibit specific symmetries.
In the following, we analyze how these properties are reflected in the harmonic coefficients and, equivalently, in the Fourier coefficients of the corresponding DFS transform.
Exploiting these relations reduces storage requirements and accelerates the NSOFT-algorithms.

By the BMC property of the DFS function, the Fourier coefficients $\ghat$ satisfy $\ghat{k}{j}{l}=(-1)^{k+l}\,\ghat{k}{-j}{l}$, as stated earlier in \Cref{lem:SymGhatDFSMethod}.
The next lemma addresses further symmetry properties specific to real-valued functions.

\begin{lemma}\label{lem:SymmetryPropertyRealValued}
  Let $N\in\IN$ and $f\in\bandlimFunc{N}$. Moreover let $\fhat\in\IC^{\abs{\J{N}}}$ and $\ghat=\coeffTrafo{N} \fhat$ be given. Then the following are equivalent:
  \begin{enumerate}[label=(\roman*)]
    \item $f$ is real-valued,
    \item $\fhat{n}{k}{l}=(-1)^{k+l} \, \conj{\fhat{n}{-k}{-l}}$ for all  $(n,k,l)\in\J{N}$,
    \item $\ghat{k}{j}{l} =  \conj{\ghat{-k}{-j}{-l}}$ for all $(k,j,l)\in \set{-N,\dots,N}{}^{3} $.
  \end{enumerate}
\end{lemma}
\begin{proof} If $f$ is real-valued, its DFS-transform $g$ is also real-valued.

  \vspace{0.2cm}\noindent\underline{$(i)\Leftrightarrow(iii)$:} This is a standard property of Fourier series, see~\cite{Plonka2018}.

  \vspace{0.2cm}\noindent\underline{$(i)\implies(ii)$:}
  Using that the Wigner-d functions are real-valued and satisfy the symmetry property $\Wignerd{n}{k}{l}(x) = (-1)^{k+l} \, \Wignerd{n}{-k}{-l}(x)$ (see~\cite{Varshalovich1988}), it follows that
  \begin{equation}\label{eq:WignerDconjSymmetryProperty}
    \conj{\WignerD{n}{-k}{-l}(\mat R(\alpha,\beta,\gamma))} = \e^{-\i k \alpha}\, \conj{\Wignerd{n}{-k}{-l}(\cos\beta)} \,\e^{-\i l \gamma} = (-1)^{k+l}\,\WignerD{n}{k}{l}(\mat R(\alpha,\beta,\gamma)).
  \end{equation}
  Hence, the harmonic coefficients satisfy
  \[\fhat{n}{k}{l} = \scp{f}{\WignerD{n}{k}{l}}_{\L2SO3} = (-1)^{k+l}\,\conj{\scp{\conj f}{\WignerD{n}{-k}{-l}}_{\L2SO3}}\]
  which yields the assumption, since $f$ is real-valued.

  \vspace{0.2cm}\noindent\underline{$(ii)\implies(i)$:}
  By \cref{eq:WignerDconjSymmetryProperty} follows
  \[ 0 = (-1)^{k+l} \, \conj{\fhat{n}{-k}{-l}}-\fhat{n}{k}{l} = (-1)^{k+l} \, \conj{\scp{f}{\WignerD{n}{-k}{-l}}} - \scp{f}{\WignerD{n}{k}{l}} = \scp{\conj f-f}{\WignerD{n}{k}{l}} \]
  for all $(n,k,l)\in\J{N}$. Hence we obtain $\conj f-f=0$, since $\conj{f}-f\in\bandlimFunc{N}$.
\end{proof}

This symmetry property allows us to halve the length of the Fourier series used in the NFFT or FFT when computing the $\SO3$-Fourier transform.
To exploit this, we split the Fourier series and reorder the summation, yielding
\begin{equation}\label{eq:RealValuedFourierSeries}
  \sum_{k,j,l=-N}^{N}\ghat{k}{j}{l}\,\e^{\i(k,j,l)\cdot(\alpha,\beta,\gamma)^{\top}} = \Re \braces*{\sum_{k,l=-N}^{N} \sum_{j=0}^{N} (1+\chi_{j\neq0}) \, \ghat{k}{j}{l} \, \e^{\i(k,j,l)\cdot(\alpha,\beta,\gamma)^{\top}} }.
\end{equation}

An additional symmetry property is established in the following lemma.

\begin{lemma}\label{lem:SymmetryPropertyAntipodal}
  Let $N\in\IN$ and $f\in\bandlimFunc{N}$. Moreover let $\fhat\in\IC^{\abs{\J{N}}}$ and $\ghat=\coeffTrafo{N} \fhat$ be given. Then the following are equivalent:
  \begin{enumerate}[label=(\roman*)]
    \item $f$ satisfies \(f(\mat R)=f(\mat R^{-1})\) for allmost all $\mat R\in\SO3$,
    \item $\fhat{n}{k}{l} = (-1)^{k+l} \, \fhat{n}{-l}{-k}$ for all $(n,k,l)\in\J{N}$,
    \item $\ghat{k}{j}{l} =  \ghat{-l}{-j}{-k}$ for all $(k,j,l)\in\set{-N,\dots,N}{}^{3}$.
  \end{enumerate}
\end{lemma}
\begin{proof} The Euler angles of the inverse rotation satisfy
  \begin{equation}\label{eq:InverseRotation}
    \mat R^{-1}(\alpha,\beta,\gamma) = \mat R(-\gamma,-\beta,-\alpha) = \mat R (\pi-\gamma,\beta,\pi-\alpha).
  \end{equation}

  \vspace{0.2cm}\noindent\underline{$(i)\implies(ii)$:}
  Using the symmetry property $\Wignerd{n}{k}{l}(x) = \Wignerd{n}{-l}{-k}(x)$ (see~\cite{Varshalovich1988}), we obtain
  \[\WignerD{n}{k}{l}(\mat R^{-1}) = (-1)^{k+l}\,\e^{-\i l \alpha} \, \Wignerd{n}{k}{l}(\cos\beta)  \, \e^{-\i k \gamma} = (-1)^{k+l}\,\WignerD{n}{-l}{-k}(\mat R).\]
  Hence, the harmonic coefficients satisfy
  \[\fhat{n}{k}{l} = \scp{f(\mat R^{-1})}{\WignerD{n}{k}{l}(\mat R^{-1})}_{\L2SO3} = (-1)^{k+l}\, \scp{f(\mat R)}{\WignerD{n}{-l}{-k}(\mat R)}_{\L2SO3},\]
  which yields the assumption.

  \vspace{0.2cm}\noindent\underline{$(ii)\implies(iii)$:}
  Substituting $(ii)$ into \cref{eq:WignerCoeffTrafo} and exploiting the symmetry properties of the Wigner-d functions (see~\cite{Varshalovich1988}) immediately yields the result.

  \vspace{0.2cm}\noindent\underline{$(iii)\implies(i)$:}
  Using \cref{eq:InverseRotation}, the Fourier series expansion of the DFS transform of $f(\mat R^{-1})$ can be written as
  \[f(\mat R^{-1}) = \sum_{k,j,l_=-N}^{N} \ghat{k}{j}{l}\, \e^{\i(k,j,l)\cdot(-\gamma,-\beta,-\alpha)^{\top}} = \sum_{k,j,l=-N}^{N} \ghat{k}{j}{l}\, \e^{\i(-l,-j,-k)\cdot(\alpha,\beta,\gamma)^{\top}}.\]
  Reordering the summation and substituting the Fourier coefficients according to property~$(iii)$, the right-hand side recovers the DFS transform of $f$.
\end{proof}

Especially in the context of crystallography, rotation-dependent functions are often invariant under a finite subgroup $S_{L}\subset\SO3$.
When represented via the DFS transform, these symmetries appear on the 3-torus as even/odd symmetries or $\tfrac{2\pi}{n}$-periodicity (for some $n\in\IN$),
and naturally induce analogous relations among the Fourier coefficients.
\begin{definition}
  Let $S_{R}$ and $S_{L}$ be finite subgroups of $\SO3$. A function $f\colon\SO3\to\IC$ is said to have right symmetry $S_{R}$ and left symmetry $S_{L}$ if
  \[ f(\mat R) = f(\mat s_{R} \cdot \mat R \cdot \mat s_{L} ) \]
  for all $\mat s_{R}\in S_{R}$ and $\mat s_{L}\in S_{L}$.
\end{definition}
Note that left and right symmetries do not generally coincide, due to the non-commutativity of rotation composition.
However, the left and right symmetry groups are identical if the function satisfies the inversion symmetry property of \Cref{lem:SymmetryPropertyAntipodal}.
A complete list of all finite symmetry groups on $\SO3$ is provided in \Cref{tab:Symmetries},
while the icosahedral group rarely occurs in crystallography due to its fivefold rotational symmetry, which is incompatible with periodic crystal lattices.

\begin{table}[h]
\renewcommand{\arraystretch}{1.5}
\begin{tabularx}{\textwidth}{|l|c|X|c|}
\hline
\multicolumn{2}{|c|}{Finite symmetry groups} & \multicolumn{1}{c|}{Representative Set} & \multicolumn{1}{c|}{Cardinality} \\
\hline
Cyclic group 	& $C_r$ & $\set{\mat R_z(\tfrac{2\pi\,s}{r})}{s=0,\dots,r-1}$ & $r$  \\
Dihedral group 			& $D_r$	& $\set{C_r \, \mat R_y(\pi\,s)}{s=0,1}$ & $2r$  \\
Tetrahedral group 	& $T$		& $\set{ D_2 \, \mat R_{\vec{1}}(\tfrac{2\pi\,s}{3})}{s=0,1,2}$ & $12$  \\
Octahedral group 		& $O$		& $\set{ D_4 \, \mat R_{\vec{1}}(\tfrac{2\pi\,s}{3})}{s=0,1,2}$ & $24$ \\
Icosahedral group & $I$ & $\set{ \mat R_{\eta}(\tfrac{2\pi\,s}{3})\,D_{5}\,\mat R_{\eta}(\tfrac{2\pi\,t}{3}) }{s,t=0,1,2}$  \newline with $\eta=\vek{\varphi^{2},0,\varphi+\sqrt{1+\varphi^{2}}\sin\tfrac{2\pi}5}^{\top}$ \newline and golden ratio $\varphi=\tfrac{1+\sqrt5}2$  & $60$ \\
\hline
\end{tabularx}
\caption{List of all finite symmetry groups on $\SO3$ ($r\in\IN$).}
\label{tab:Symmetries}
\end{table}

Similar to factor sets, we use the previous definition to introduce the double coset space
\begin{equation*}
  \lcrnicefrac{S_R}{\SO3}{S_L}=\set{S_R\,\mat R\,S_L }{ \mat R \in\SO3}
\end{equation*}
where $S_{R}$ and $S_{L}$ are finite subgroups of $\SO3$.

\begin{changedbox}
It is important to note that the double coset space is, in general, neither a group nor a smooth manifold.

However, if only a left or only a right symmetry group is imposed, the resulting quotient is a smooth manifold and can be regarded as a homogeneous space.
In this case, the quotient structure can be understood within the DFS framework as well, now with not just a double, but a multi-fold coverage.
\end{changedbox}

Since the left and right symmetry groups can be classified as shown in \Cref{tab:Symmetries}, additional symmetry properties can be derived.

\begin{theorem}\label{thm:CrystalSymmetryProperties}
  Let $S_{R}$ and $S_{L}$ be finite subgroups of $\SO3$, and $f\in\bandlimFunc{N}{\lcrnicefrac{S_R}{\SO3}{S_L}}$ with $N\in\IN$.
  Furthermore, let $\fhat\in\IC^{\abs{\J{N}}}$ and $\ghat=\coeffTrafo{N} \fhat$ be given.
  Then, for any $r\in\IN$ it yields:
  \begin{enumerate}[label=\alph*)]
    \item The following are equivalent:
          \begin{enumerate}[label=(\roman*),ref={[a)(\roman*)]}]
            \item $C_{r} \subseteq S_{R}$
            \item\label{prop:a2} If $k \mod r \neq 0$ then $\fhat{n}{k}{l}=0$ for all $(n,k,l)\in\J{N}$.
            \item\label{prop:a3} If $k \mod r \neq 0$ then $\ghat{k}{j}{l}=0$ for all $(k,j,l)\in\set{-N,\dots,N}{}^{3}$.
          \end{enumerate}

    \item The following are equivalent:
          \begin{enumerate}[label=(\roman*)]
            \item $D_{r} \subseteq S_{R}$
            \item Property \ref{prop:a2} and $\fhat{n}{k}{l} = (-1)^{n+k} \, \fhat{n}{-k}{l}$ for all $(n,k,l)\in\J{N}$.
            \item Property \ref{prop:a3} and $\ghat{k}{j}{l} = (-1)^{j} \, \ghat{-k}{j}{l}$ for all $(k,j,l)\in\set{-N,\dots,N}{}^{3}$.
          \end{enumerate}
    \item The following are equivalent:
          \begin{enumerate}[label=(\roman*),ref={[c)(\roman*)]}]
            \item $C_r\subseteq S_L$
            \item\label{prop:c2} If $l \mod r \neq 0$ then $\fhat{n}{k}{l}=0$ for all $(n,k,l)\in\J{N}$.
            \item\label{prop:c3} If $l \mod r \neq 0$ then $\ghat{k}{j}{l}=0$ for all $(k,j,l)\in\set{-N,\dots,N}{}^{3}$.
          \end{enumerate}
    \item The following are equivalent:
          \begin{enumerate}[label=(\roman*)]
            \item $D_r\subseteq S_L$
            \item Property \ref{prop:c2} and $\fhat{n}{k}{l} = (-1)^{n+l} \, \fhat{n}{k}{-l}$ for all $(n,k,l)\in\J{N}$.
            \item Property \ref{prop:c3} and $\ghat{k}{j}{l} = (-1)^{j} \, \ghat{k}{j}{-l}$ for all $(k,j,l)\in\set{-N,\dots,N}{}^{3}$.
          \end{enumerate}
  \end{enumerate}
\end{theorem}
\begin{proof}
  We will only proof \textit{b)}. The other cases work analogous.

  \vspace{0.2cm}\noindent\underline{$(i)\implies(ii)$:}
  Let $S_{R}=D_{r}$ and $S_{L}=\set{\id}{}$.
  The harmonic coefficients of $f$, with respect to the $\L2SO3$-norm satisfy
  \[ \fhat{n}{k}{l} = \scp{f}{\WignerD{n}{k}{l}}_{\L2SO3} = \sum_{\mat P\in D_{r}} \int_{\lcrnicefrac{D_{r}~}{\;\,\SO3\;}{~\set{\id}{}}} f(\mat P \mat R)\, \conj{\WignerD{n}{k}{l}(\mat P \mat R)} \d{\mat R} \]
  for all $(n,k,l)\in\J{N}$.
  Using the symmetry of $f$ and the representation property~\eqref{eq:RepresentationProperty}, we obtain
  \begin{align*}
    \fhat{n}{k}{l} &= \sum_{\mat P\in D_{r}} \sum_{u=-n}^{n} \conj{\WignerD{n}{k}{u}(\mat P)} \,  \int_{\lcrnicefrac{D_{r}~}{\;\,\SO3\;}{~\set{\id}{}}} f(\mat R)\, \conj{\WignerD{n}{u}{l}(\mat R)} \d{\mat R} \\
    &= \sum_{u=-n}^{n} \braces*{\sum_{\mat P\in D_{r}}\conj{\WignerD{n}{k}{u}(\mat P)}} \cdot  \scp{f}{\WignerD{n}{u}{l}}_{\Lp{2}(\lcrnicefrac{D_{r}~}{\;\,\SO3\;}{~\set{\id}{}})} .
  \end{align*}
  By the definition of the Wigner-D functions, it follows
  \begin{align*}
    \sum_{\mat P\in D_{r}}\conj{\WignerD{n}{k}{u}(\mat P)} &= \braces*{\sum_{s=0}^{r-1} \e^{-2\pi\i\,k\,s/r}}\braces*{\Wignerd{n}{k}{u}(1)+\Wignerd{n}{k}{u}(-1)} \\
                                                        &= r\cdot\charFunc{r\,\nmid\, k}\cdot\braces*{\charFunc{k=u}+(-1)^{n+k}\cdot\charFunc{-k=u}}
  \end{align*}
  which immediately yields $(ii)$.

  \vspace{0.2cm}\noindent\underline{$(ii)\implies(iii)$:}
  Substituting $(ii)$ into \cref{eq:WignerCoeffTrafo} and exploiting the symmetry properties of the Wigner-d functions (see~\cite{Varshalovich1988}) immediately yields the result.

  \vspace{0.2cm}\noindent\underline{$(iii)\implies(i)$:}
  Obviously,
  \[ D_{r}=\set{\mat R(\tfrac{2\pi s}{r},\pi,0)}{s=0,\dots,r-1;\,t=0,1}. \]
  For arbitrary $s$ and $t$ we obtain
  \[ \mat R(\tfrac{2\pi s}{r},t\pi,0) \, \mat R(\alpha,\beta,\gamma) = \mat R(\tfrac{2\pi s}{r}-\alpha,\beta+t\pi,\gamma). \]
  Using this identity, the Fourier series expansion of the DFS transform of $f$ at $\mat P \cdot \mat R$ with $\mat P \in D_{r}$ reads as
  \begin{align*}
    f(\mat R(\tfrac{2\pi s}{r}-\alpha,\beta+t\pi,\gamma)) &= \sum_{k,j,l=-N}^{N} \ghat{k}{j}{l}\, \e^{\i(k,j,l)\cdot(\tfrac{2\pi s}{r}-\alpha,\beta+t\pi,\gamma)^{\top}} \\
                                                         &= \sum_{k,j,l=-N}^{N} \e^{2\pi\i ks/r}\,(-1)^{tj}\,\ghat{k}{j}{l}\, \e^{\i(-k,j,l)\cdot(\alpha,\beta,\gamma)^{\top}}.
  \end{align*}
  By property~\ref{prop:a3} we already know that $\e^{2\pi\i ks/r}\,\ghat{k}{j}{l}=\ghat{k}{j}{l}$. Moreover, reordering the summation and substituting the Fourier coefficients according to property~$(iii)$ shows that the right-hand side coincides with the DFS transform of $f$ at $\mat R$.
  Hence, $f(\mat P \, \mat R)=f(\mat R)$ for all $\mat P \in D_{r}$.
\end{proof}

Overall, the symmetry properties established in \Cref{lem:SymGhatDFSMethod}, \Cref{lem:SymmetryPropertyRealValued}, \Cref{lem:SymmetryPropertyAntipodal}, and \Cref{thm:CrystalSymmetryProperties} can be leveraged in four ways:
\begin{itemize}
  \item Since many harmonic coefficients either vanish or coincide up to a sign in their real and imaginary parts, the symmetry properties substantially reduce the disk storage requirements. Specifically if $S_{L}$ and $S_{R}$ are cyclic or dihedral, the compression factor is
        \[ c_{f} = (1+\charFunc{f \text{ is real\,}})\cdot(1+\charFunc{f(\mat R) = f(\mat R^{-1}) \text{ for all } \mat R \in \SO3})\cdot\abs{S_{L}}\cdot\abs{S_{R}}. \]

  \item The (direct) Wigner transform $\coeffTrafo{N}$ and its adjoint $\coeffTrafo{N}^{H}$ speed up by a factor $c_{f}$, as only one representative of each symmetry class of harmonic/Fourier coefficients must be computed.

  \item The symmetry properties of the Fourier coefficients $\ghat$ reduce the effective size of the discrete Fourier transform to $\tfrac{2N+1}{r} \times (2N+1) \times \tfrac{2N+1}{s}$, where $r$ and $s$ are the orders of the underlying cyclic groups.
  For dihedral groups, the transform can be further reduced by splitting it into cosine and sine parts.

  \item The inverse $\SO3$-Fourier transform, see \Cref{sec:InverseNSOFT}, requires function values only on one representative quadrature node per symmetry class. For optimal efficiency, the bandwidth $N$ should satisfy $r,s \mid (2N+2)$ for the left/right groups $C_{r},D_{r}$ and $C_{s},D_{s}$.
\end{itemize}


\subsection{Fast Realizations of the \texorpdfstring{$\coeffTrafo{N}$}{W_N}-Operator}\label{sec:WignerTrafoAlgorithms}

The Wigner transform $\coeffTrafo{N}$ introduced in \Cref{def:DFS-TrafoW} is a coefficient transform mapping harmonic to Fourier coefficients and is therefore independent of the evaluation points used in the NSOFT.
In this section, we outline two common algorithmic realizations.
A detailed numerical comparison will be presented in \Cref{sec:Numerics}.

The First approach is the \emph{Wigner transform via fast polynomial transform (FPT)}~\cite{Potts2009,Kostelec2008}, previously introduced in \Cref{rem:WignerTrafoPotts}.
It achieves a complexity of $\mathcal O(N^{3}\log^{2}N)$ but suffers from numerical instabilities~\cite{Potts2003}, which can be mitigated by a stabilization step proposed in~\cite{Potts1998}.

A simpler but asymptotically slower alternative is the \emph{direct Wigner transform}, obtained by implementing \Cref{eq:WignerCoeffTrafo} for all index triples $k,j,l=-N,\dots,N$~\cite{Bunge1982,Risbo1996}. Its cost is $\mathcal O(N^{4})$ flops.
Despite its higher complexity, the method has two practical advantages.

First, its simplicity makes the incorporation of symmetry reductions straightforward, see \Cref{sec:Symmetries}.

Second, while the Wigner transform via FPT requires the costly precomputation of all Wigner-d matrices up to degree $N$ at $N+1$ nodes, the direct Wigner transform only needs the special values $\Wignerd{n}{k}{l}(0)$ at $x=0$, which can be obtained from recurrence relations based on Jacobi polynomials.
These recurrences are known to be numerically unstable~\cite{Dachsel2006,Feng2015,Allen2019,Wang2022}, though see~\cite{Gumerov2015} for a weakly unstable variant suitable for large bandwidth.
Nevertheless, our numerical experiments in \Cref{sec:Numerics} demonstrate that the error remains manageable.

\begin{remark}
  We implemented the direct Wigner transform in the MATLAB toolbox \texttt{MTEX}~\cite{Hielscher2007} as a C++ script, which processes the data in a linear, cache-friendly order.
  Furthermore, iterating over the bandwidth $n=0,\dots,N$ allows the Wigner-d functions to be updated on the fly, eliminating the need to keep all values in storage simultaneously.
\end{remark}


\section{Numerical Experiments}\label{sec:Numerics}

In this chapter, we present a numerical analysis of the two algorithms for the Wigner transform, described in \Cref{sec:WignerTrafoAlgorithms}:
\begin{enumerate}
  \item Direct Wigner transform (see \cref{eq:WignerCoeffTrafo})
  \item Wigner transform via fast polynomial transform (FPT) \cite{Potts2009}
\end{enumerate}
We will demonstrate that, in practice, the direct Wigner transform is faster, simpler, and more accurate than the FPT-based approach, even though its theoretical complexity is higher.

All algorithms were implemented in \texttt{C} and tested on a $3.8\,\mathrm{GHz}$ AMD Ryzen$^{\text{TM}}$ system with $128\,\mathrm{GB}$ of RAM, using double-precision arithmetic.
The implementations rely on the  FFTW 3.3.10~\cite{Frigo2021}, NFFT 3.5.3~\cite{Keiner2009}, and MTEX 6.1~\cite{Hielscher2007} libraries.
Note that both algorithms allow for parallelization and have been implemented accordingly.

\subsection{Running Time}
The direct Wigner transform has higher asymptotic complexity and is therefore theoretically much slower than the FPT-based method for large bandwidths.
In practice, however, such bandwidths are difficult to reach, since the three-dimensional setting causes cubic growth in both the number of harmonic coefficients and the runtime.
Our numerical experiments, illustrated in \Cref{fig:NSOFT_Time}, indicate that for bandwidths below $256$, the direct Wigner transform outperforms the FPT-based approaches. This is partly because modern computing architectures have significantly sped up direct matrix-vector multiplications.

In \Cref{fig:NSOFT_Time}, we explicitly distinguish between the fast Wigner transform with and without precomputations.
These expensive precomputations, involving roughly $16N^3\log N$ evaluations of adapted Wigner-d functions~\cite{Potts2009}, are required for the FPT and need to be kept in memory, which further slows down the method.
Consequently, we could not run the FPT-based approach for $N>256$ due to memory limitations.
Furthermore, the precomputations depend on the bandwidth and must be redone, whenever the NSOFT is computed for a different $N$.

Nevertheless, when computing the $\SO3$ Fourier transform, the Wigner transform is combined with an NFFT, which makes the CPU time of the entire algorithm ultimately limited by the NFFT.

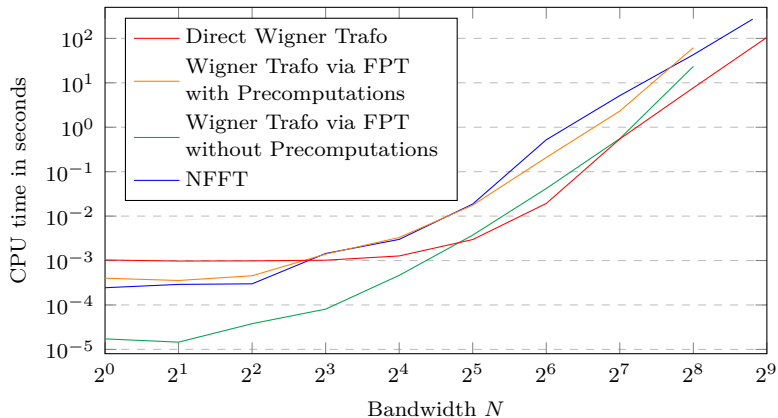
\begin{figure}[H]
  \centering
  \begin{subfigure}[c]{0.7\textwidth}
    \begin{tikzpicture}
      \begin{axis}[
        width = \textwidth,
        height = 0.6\textwidth,
        xmin=1, xmax=512,
        ymin=8*10^(-6), ymax=500,
        ymode=log, xmode=log, log basis x={2},
        xtick={1,2,4,8,16,32,64,128,256,512},
        ytick={10^-5,10^(-4),0.001,0.01,0.1,1,10,100},
        minor tick style = {draw=none},
        ymajorgrids,
        grid style = {dashed, gray!30},
        font = {\scriptsize},
        legend style={text width=3.25cm,align=left,xshift=-3pt},
        legend pos=north west,
        reverse legend,
        xlabel={Bandwidth $N$},
        ylabel={CPU time (s)},
        at={(0,0)},
        anchor=center,
        clip=false
        ]
        \addplot[color=blue,dotted,line width=1.2pt] table [x=bandwidth, y=NFFT3_N3_m4_sigma15, col sep=comma]{data/NSOFTRunningTime.txt};
        \addlegendentry{NFFT};
        \addplot[color=Green,dashdotted,line width=0.8pt] table [x=bandwidth, y=fastWCT_withoutPre, col sep=comma]{data/NSOFTRunningTime.txt};
        \addlegendentry{\shortstack[l]{Wigner Trafo via FPT \\ without Precomputations}};
        \addplot[color=orange,dashed,line width=0.8pt] table [x=bandwidth, y=fastWCT, col sep=comma]{data/NSOFTRunningTime.txt};
        \addlegendentry{\shortstack[l]{Wigner Trafo via FPT \\ with Precomputations}};
        \addplot[color=red,solid,line width=0.8pt] table [x=bandwidth, y=directWCT, col sep=comma]{data/NSOFTRunningTime.txt};
        \addlegendentry{Direct Wigner Trafo};
      \end{axis}
    \end{tikzpicture}
  \end{subfigure}
  \caption{Comparison of the CPU times, that are required to compute the direct Wigner transform, the Wigner transform via FPT and the NFFT (oversampling factor $\sigma=1.5$, cut-off parameter $m=4$, Kaiser-Bessel window function, $N^3$ nodes).}\label{fig:NSOFT_Time}
\end{figure}

A major advantage of the direct Wigner transform lies in its simpler implementation, which makes it much easier to exploit the symmetry properties from \Cref{sec:Symmetries} for further runtime reduction.

\subsection{Accuracy}
We now assess the stability of both algorithms for various bandwidths $N\in\IN$.
For this purpose, we randomly generate harmonic coefficient vectors $\bm{\hat f}\in\IC^{\abs{\J{N}}}$ with entries uniformly distributed on the complex unit disk.

The corresponding band-limited function is then evaluated on the Clenshaw-Curtis quadrature grid $\mathcal R\subset \SO3$, and the harmonic coefficients are reconstructed via numerical quadrature from these sample points.

To quantify the accuracy of the Wigner transform implementation $\coeffTrafo{N}$, we measure the relative error
\begin{align*}
  E_{\ell_{1}\to\ell_{2}} = \max_{i=1,\dots,100} \frac{\norm{ \fhat_{i} - \bm{D}^{-1}_{\mathcal R_{N},N}  \bm{D}_{\mathcal R_{N},N}  \fhat_{i} }_{\ell_{2}}}{\norm{\fhat_{i}}_{{\ell_{1}}}} = \max_{i=1,\dots,100} \frac{\norm{ \fhat_{i} - \coeffTrafo{N}^{H} \, (\FourierTrafo{\mathcal R,N}^{H}  \bm\Lambda \FourierTrafo{\mathcal R,N} ) \, \coeffTrafo{N} \fhat_{i} }_{\ell_{2}} }{\norm{\fhat_{i}}_{{\ell_{1}}}},
\end{align*}
where $\FourierTrafo{\mathcal R,N}$ denotes the equispaced Fourier matrix and $\Lambda$ is the diagonal matrix of Clenshaw-Curtis quadrature weights.

Since the Fourier matrix is orthogonal up to a scaling factor, the condition number of
$\FourierTrafo{\mathcal R_{N},N}^{H}  \bm\Lambda \FourierTrafo{\mathcal R_{N},N}$ is approximately $2\pi N$, reflecting the ratio between the largest and smallest quadrature weights.
Consequently, the condition number of the Wigner transform satisfies
\[ \kappa(\coeffTrafo{N}) = \sqrt{2\pi N}. \]

In~\Cref{fig:WCT_Error}, we compare the relative errors of the two algorithms.
While the FPT error grows with bandwidth, the direct Wigner transform becomes more accurate.

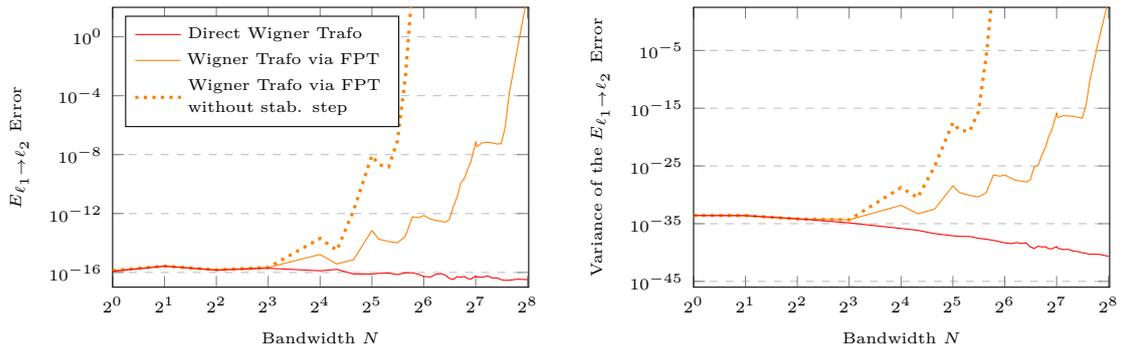
\begin{figure}[H]
  \centering
  \begin{subfigure}[c]{0.48\textwidth}
    \begin{tikzpicture}
      \begin{axis}[
        width = \textwidth,
        height = 0.75\textwidth,
        xmin=1, xmax=260,
        ymin=1*10^(-17), ymax=10^(2),
        ymode=log, xmode=log, log basis x={2},
        xtick={1,2,4,8,16,32,64,128,256},
        ytick={10^(-16),10^(-12),10^(-8),10^(-4),10^(0),10^(4)},
        minor tick style = {draw=none},
        ymajorgrids,
        grid style = dashed,
        font = {\tiny},
        legend style={text width=2.6cm,align=left,xshift=-2pt},
        legend pos=north west,
        reverse legend,
        ylabel={$E_{\ell_{1}\to\ell_{2}}$ Error},
        xlabel={Bandwidth $N$}
        ]
        \addplot[color=black,dotted,line width=1.2pt] table [x=bw, y=FWTwoStab, col sep=comma]{data/Error_fuaWCT.txt};
        \addlegendentry{\shortstack[l]{Wigner Trafo via FPT \\ without stab. step}};

        \addplot[color=orange,dashed,line width=1pt] table [x=bw, y=FWT, col sep=comma]{data/Error_fuaWCT.txt};
        \addlegendentry{Wigner Trafo via FPT};

        \addplot[color=red,line width=0.8pt] table [x=bw, y=directWCT, col sep=comma]{data/Error_fuaWCT.txt};
        \addlegendentry{Direct Wigner Trafo};
      \end{axis}
    \end{tikzpicture}
  \end{subfigure}
  \hfill
  \begin{subfigure}[c]{0.48\textwidth}
    \begin{tikzpicture}
      \begin{axis}[
        width = \textwidth,
        height = 0.75\textwidth,
        xmin=1, xmax=260,
        ymin=1*10^(-46), ymax=3*10^(2),
        ymode=log, xmode=log, log basis x={2},
        xtick={1,2,4,8,16,32,64,128,256},
        ytick={10^(-45),10^(-35),10^(-25),10^(-15),10^(-5),10^(5)},
        minor tick style = {draw=none},
        ymajorgrids,
        grid style = dashed,
        font = {\tiny},
        legend style={text width=2.5cm,align=left},
        legend pos=north west,
        reverse legend,
        ylabel={Variance of the $E_{\ell_{1}\to\ell_{2}}$ Error},
        xlabel={Bandwidth $N$}
        ]
        \addplot[color=black,dotted,line width=1.2pt] table [x=bw, y=FWTwoStab, col sep=comma]{data/Variance_Error_fuaWCT.txt};

        \addplot[color=orange,dashed,line width=1pt] table [x=bw, y=FWT, col sep=comma]{data/Variance_Error_fuaWCT.txt};

        \addplot[color=red,line width=0.8pt] table [x=bw, y=directWCT, col sep=comma]{data/Variance_Error_fuaWCT.txt};
      \end{axis}
    \end{tikzpicture}
  \end{subfigure}
  \caption{Investigation of the accuracy of different Wigner transform implementations depending on the bandwidth. The left panel shows the relative error $E_{\ell_{1}\to\ell_{2}}$ for a randomly chosen harmonic coefficient vector $\fhat$, while the right panel displays the variance of this error, estimated from 100 independent random coefficient vectors.}\label{fig:WCT_Error}
\end{figure}

\begin{changedbox}
The implementation of the FPT in the context of Wigner-d functions becomes numerically unstable for higher bandwidths.
The main reason is that, for certain index quadruples $(n,k,l,c)$, the absolute values of the associated Wigner-d functions $\Wignerd{n}{k}{l}(x;c)$, defined in~\cite{Potts2009}, with shift parameter $c\in\IN_{0}$, become large.
During the cascade summation used in the FPT, these values are multiplied by small values of $\Wignerd{n}{k}{l}(x)$, which leads to the numerical instabilities due to cancellation and error amplification. A related effect is discussed in detail for associated Legendre functions in~\cite{Potts1998}. To address this issue, the authors describe a stabilization technique that replaces problematic multiplication steps in the FPT whenever $\abs{\Wignerd{n}{k}{l}(x;c)}$ exceeds a threshold $\kappa$. This improves stability at the expense of increased computational cost.
However, even with this stabilization technique (with threshold $\kappa=10^{3}$), the algorithm fails for high bandwidths $N>100$.
\end{changedbox}


\begin{addedbox}
\subsection*{Code Availability}
All numerical experiments have been realized using the Matlab Toolbox MTEX. The cor-
responding scripts are available at \url{https://github.com/mtex-toolbox/mtex-paper/tree/master/DFSMethodonSO3}.
\end{addedbox}

\bibliography{literatur}

@Book{Varshalovich1988,
  author    = {Varshalovich, Dmitry Aleksandrovich and Moskalev, Anatoly Nikolayevich and Khersonskii, Valerii Kelmanovich},
  publisher = {World Scientific Publishing Co., Inc., Teaneck, NJ},
  title     = {Quantum theory of angular momentum},
  year      = {1988},
  isbn      = {9971-50-107-4},
  note      = {Translated from the Russian},
  doi       = {10.1142/0270},
  mrnumber  = {1022665},
  pages     = {xii+514},
}

@Article{Potts2009,
  author   = {Potts, Daniel and Prestin, Jürgen and Vollrath, Antje},
  journal  = {Numerical Algorithms},
  title    = {A fast algorithm for nonequispaced {F}ourier transforms on the rotation group},
  year     = {2009},
  issn     = {1017-1398},
  number   = {3},
  pages    = {355--384},
  volume   = {52},
  doi      = {10.1007/s11075-009-9277-0},
  mrnumber = {2563947},
}

@Article{Potts2003,
  author   = {Potts, Daniel and Steidl, Gabriele and Tasche, Manfred},
  journal  = {Journal of Concrete and Applicable Mathematics},
  title    = {Numerical stability of fast trigonometric transforms - a worst case study},
  year     = {2003},
  issn     = {1548-5390},
  number   = {1},
  pages    = {1--35},
  volume   = {1},
  mrnumber = {2132664},
}

@Book{Plonka2018,
  author    = {Plonka, Gerlind and Potts, Daniel and Steidl, Gabriele and Tasche, Manfred},
  publisher = {Birkh\"{a}user/Springer, Cham},
  title     = {Numerical {F}ourier analysis},
  year      = {2018},
  isbn      = {978-3-030-04305-6},
  series    = {Applied and Numerical Harmonic Analysis},
  doi       = {10.1007/978-3-030-04306-3},
  mrnumber  = {3890075},
  pages     = {xvi+168},
}

@Article{Kostelec2008,
  author   = {Kostelec, Peter J. and Rockmore, Daniel N.},
  journal  = {The Journal of Fourier Analysis and Applications},
  title    = {F{FT}s on the rotation group},
  year     = {2008},
  issn     = {1069-5869},
  number   = {2},
  pages    = {145--179},
  volume   = {14},
  doi      = {10.1007/s00041-008-9013-5},
  mrnumber = {2383721},
}

@Article{Khalid2015,
  author        = {Zubair Khalid and Salman Durrani and Rodney A. Kennedy and Yves Wiaux and Jason D. McEwen},
  journal       = {IEEE Signal Processing Letters, Vol. 23, No. 2, February 2016},
  title         = {Gauss-{L}egendre Sampling on the Rotation Group},
  year          = {2015},
  month         = aug,
  archiveprefix = {arXiv},
  doi           = {10.1109/LSP.2015.2503295},
  eprint        = {1508.03353},
  primaryclass  = {cs.IT},
}

@Book{Szegoe1975,
  author    = {Szeg\H{o}, G\'{a}bor},
  publisher = {American Mathematical Society, Providence, R.I.},
  title     = {Orthogonal polynomials},
  year      = {1975},
  edition   = {Fourth},
  series    = {American Mathematical Society Colloquium Publications, Vol. XXIII},
  mrnumber  = {0372517},
  pages     = {xiii+432},
}

@Article{Graef2009,
  author   = {Gräf, Manuel and Potts, Daniel},
  journal  = {Numerical Functional Analysis and Optimization. An International Journal},
  title    = {Sampling sets and quadrature formulae on the rotation group},
  year     = {2009},
  issn     = {0163-0563},
  number   = {7-8},
  pages    = {665--688},
  volume   = {30},
  doi      = {10.1080/01630560903163508},
  mrnumber = {2555654},
}

@Article{Risbo1996,
  author    = {Torben Risbo},
  journal   = {Journal of Geodesy},
  title     = {Fourier transform summation of {L}egendre series and {D}-functions},
  year      = {1996},
  month     = jul,
  number    = {7},
  pages     = {383--396},
  volume    = {70},
  doi       = {10.1007/bf01090814},
  publisher = {Springer Science and Business Media {LLC}},
}

@Book{Vilenkin2012,
  author    = {Vilenkin, Naum Yakovlevich and Klimyk, Anatolii U.},
  publisher = {Springer},
  title     = {Representation of {L}ie groups and special functions: Volume 1 simplest {L}ie groups, special functions and integral transforms},
  year      = {2012},
  isbn      = {9789401135382},
  subtitle  = {Simplest Lie Groups, Special Functions and Integral Transforms},
}

@Article{Mildenberger2023,
  author        = {Mildenberger, Sophie and Quellmalz, Michael},
  journal       = {Sampling Theory, Signal Processing, and Data Analysis 21, Article number: 23, 2023. (Open access)},
  title         = {A double {F}ourier sphere method for {$d$}-dimensional manifolds},
  year          = {2023},
  month         = jan,
  number        = {2},
  volume        = {21},
  archiveprefix = {arXiv},
  copyright     = {arXiv.org perpetual, non-exclusive license},
  doi           = {10.1007/s43670-023-00064-8},
  eprint        = {2301.06540},
  primaryclass  = {math.NA},
  publisher     = {Springer Science and Business Media {LLC}},
}

@Article{Feng2015,
  author    = {Feng, X. M. and Wang, P. and Yang, W. and Jin, G. R.},
  journal   = {Physical Review E},
  title     = {High-precision evaluation of {W}igner’s d-matrix by exact diagonalization},
  year      = {2015},
  issn      = {1550-2376},
  month     = oct,
  number    = {4},
  pages     = {043307},
  volume    = {92},
  doi       = {10.1103/physreve.92.043307},
  publisher = {American Physical Society (APS)},
}

@Article{Allen2019,
  author    = {Allen, Wesley D.},
  journal   = {The Journal of Chemical Physics},
  title     = {Wigner numbers},
  year      = {2019},
  issn      = {1089-7690},
  month     = dec,
  number    = {24},
  volume    = {151},
  doi       = {10.1063/1.5135721},
  publisher = {AIP Publishing},
}

@Article{Wang2022,
  author    = {Wang, Bin-Lei and Gao, Fan and Wang, Long-Jun and Sun, Yang},
  journal   = {Physical Review C},
  title     = {Effective and efficient algorithm for the {W}igner rotation matrix at high angular momenta},
  year      = {2022},
  issn      = {2469-9993},
  month     = nov,
  number    = {5},
  pages     = {054320},
  volume    = {106},
  doi       = {10.1103/physrevc.106.054320},
  publisher = {American Physical Society (APS)},
}

@Article{Dachsel2006,
  author    = {Dachsel, Holger},
  journal   = {The Journal of Chemical Physics},
  title     = {Fast and accurate determination of the {W}igner rotation matrices in the fast multipole method},
  year      = {2006},
  issn      = {1089-7690},
  month     = apr,
  number    = {14},
  volume    = {124},
  doi       = {10.1063/1.2194548},
  publisher = {AIP Publishing},
}

@Article{Mildenberger2022,
  author    = {Mildenberger, Sophie and Quellmalz, Michael},
  journal   = {Journal of Fourier Analysis and Applications},
  title     = {Approximation Properties of the Double {F}ourier Sphere Method},
  year      = {2022},
  issn      = {1531-5851},
  month     = mar,
  number    = {2},
  volume    = {28},
  doi       = {10.1007/s00041-022-09928-4},
  publisher = {Springer Science and Business Media LLC},
}

@InBook{Gumerov2015,
  author    = {Gumerov, Nail A. and Duraiswami, Ramani},
  pages     = {105--141},
  publisher = {Springer International Publishing},
  title     = {Recursive computation of spherical harmonic rotation coefficients of large degree},
  year      = {2015},
  isbn      = {9783319132303},
  booktitle = {Excursions in Harmonic Analysis, Volume 3},
  doi       = {10.1007/978-3-319-13230-3_5},
  issn      = {2296-5017},
}

@Article{Keiner2009,
  author    = {Keiner, Jens and Kunis, Stefan and Potts, Daniel},
  journal   = {ACM Transactions on Mathematical Software},
  title     = {Using {NFFT 3} - A Software Library for Various Nonequispaced Fast {F}ourier Transforms},
  year      = {2009},
  issn      = {1557-7295},
  month     = aug,
  number    = {4},
  pages     = {1--30},
  volume    = {36},
  doi       = {10.1145/1555386.1555388},
  publisher = {Association for Computing Machinery (ACM)},
}

@Article{Hielscher2010,
  author    = {Hielscher, Ralf and Prestin, Jürgen and Vollrath, Antje},
  journal   = {Mathematical Geosciences},
  title     = {Fast Summation of Functions on the Rotation Group},
  year      = {2010},
  issn      = {1874-8953},
  month     = jun,
  number    = {7},
  pages     = {773--794},
  volume    = {42},
  doi       = {10.1007/s11004-010-9281-x},
  publisher = {Springer Science and Business Media LLC},
}

@Article{Graef2011,
  author    = {Gräf, Manuel},
  journal   = {Advances in Computational Mathematics},
  title     = {A unified approach to scattered data approximation on {$S^{3}$} and {SO}(3)},
  year      = {2011},
  issn      = {1572-9044},
  month     = sep,
  number    = {3},
  pages     = {379--392},
  volume    = {37},
  doi       = {10.1007/s10444-011-9214-3},
  publisher = {Springer Science and Business Media LLC},
}

@Article{Graef2008,
  author    = {Gräf, Manuel and Kunis, Stefan},
  journal   = {Electronic Transactions on Numerical Analysis.Volume 31, pp. 30-39},
  title     = {Stability results for scattered data interpolation on the rotation group},
  year      = {2008},
  issn      = {1068-9613},
  pages     = {30-39},
  volume    = {31},
  date      = {2008},
  publisher = {Kent State University},
}

@Article{Potts2001,
  author  = {Potts, Daniel and Steidl, Gabriele and Tasche, Manfred},
  journal = {Mod. Sampl. theory},
  title   = {Fast {F}ourier transforms for nonequispaced data: a tutorial},
  year    = {2001},
  month   = {01},
  pages   = {19-25},
  volume  = {23},
}

@Book{Gradshteyn1980,
  author    = {Gradshteyn, I. S. and Ryzhik, I. M.},
  publisher = {Elsevier Science},
  title     = {Table of integrals, series, and products},
  year      = {1980},
  address   = {Burlington},
  isbn      = {978-0-12-294760-5},
  pagetotal = {1207},
  ppn_gvk   = {813304881},
}

@Article{Orszag1974,
  author  = {Steven A. Orszag},
  journal = {Monthly Weather Review},
  title   = {Fourier Series on Spheres},
  year    = {1974},
  pages   = {56-75},
  volume  = {102},
  url     = {https://api.semanticscholar.org/CorpusID:119591242},
}

@Article{Boyd1978,
  author    = {Boyd, John P.},
  journal   = {Monthly Weather Review},
  title     = {The Choice of Spectral Functions on a Sphere for Boundary and Eigenvalue Problems: A Comparison of {C}hebyshev, {F}ourier and Associated {L}egendre Expansions},
  year      = {1978},
  issn      = {1520-0493},
  month     = aug,
  number    = {8},
  pages     = {1184--1191},
  volume    = {106},
  doi       = {10.1175/1520-0493(1978)106<1184:tcosfo>2.0.co;2},
  publisher = {American Meteorological Society},
}

@Article{Yee1980,
  author    = {Yee, Samuel Y. K.},
  journal   = {Monthly Weather Review},
  title     = {Studies on {F}ourier Series on Spheres},
  year      = {1980},
  issn      = {1520-0493},
  month     = may,
  number    = {5},
  pages     = {676--678},
  volume    = {108},
  doi       = {10.1175/1520-0493(1980)108<0676:sofsos>2.0.co;2},
  publisher = {American Meteorological Society},
}

@Article{Kovacs2003,
  author    = {Kovacs, Julio A. and Chacón, Pablo and Cong, Yao and Metwally, Essam and Wriggers, Willy},
  journal   = {Acta Crystallographica Section D Biological Crystallography},
  title     = {Fast rotational matching of rigid bodies by fast {F}ourier transform acceleration of five degrees of freedom},
  year      = {2003},
  issn      = {0907-4449},
  month     = jul,
  number    = {8},
  pages     = {1371--1376},
  volume    = {59},
  doi       = {10.1107/s0907444903011247},
  publisher = {International Union of Crystallography (IUCr)},
}

@InProceedings{Makadia2003,
  author     = {Makadia, Ameesh and Daniilidis, Kostas},
  booktitle  = {IEEE Computer Society Conference on Computer Vision and Pattern Recognition, 2003. Proceedings.},
  title      = {Direct {3D}-rotation estimation from spherical images via a generalized shift theorem},
  year       = {2003},
  pages      = {II-217--24},
  publisher  = {IEEE Comput. Soc},
  series     = {CVPR-03},
  volume     = {2},
  collection = {CVPR-03},
  doi        = {10.1109/cvpr.2003.1211473},
}

@Article{Chirikjian2001,
  author    = {Chirikjian, Gregory S. and Kyatkin, Alexander B.},
  journal   = {Applied Mechanics Reviews},
  title     = {Engineering Applications of Noncommutative Harmonic Analysis: With Emphasis on Rotation and Motion Groups},
  year      = {2001},
  issn      = {2379-0407},
  month     = nov,
  number    = {6},
  pages     = {B97--B98},
  volume    = {54},
  doi       = {10.1115/1.1421108},
  publisher = {ASME International},
}

@Article{Boogaart2007,
  author    = {van den Boogaart, Karl Gerald and Hielscher, Ralf and Prestin, Jürgen and Schaeben, Helmut},
  journal   = {Journal of Computational and Applied Mathematics},
  title     = {Kernel-based methods for inversion of the {R}adon transform on {SO}(3) and their applications to texture analysis},
  year      = {2007},
  issn      = {0377-0427},
  month     = feb,
  number    = {1},
  pages     = {122--140},
  volume    = {199},
  doi       = {10.1016/j.cam.2005.12.003},
  publisher = {Elsevier BV},
}

@Article{Merilees1973,
  author    = {Merilees, Philip E.},
  journal   = {Atmosphere},
  title     = {The pseudospectral approximation applied to the shallow water equations on a sphere},
  year      = {1973},
  issn      = {0004-6973},
  month     = jan,
  number    = {1},
  pages     = {13--20},
  volume    = {11},
  doi       = {10.1080/00046973.1973.9648342},
  publisher = {Informa UK Limited},
}

@Book{Bunge1982,
  author    = {Bunge, Hans-Joachim},
  publisher = {Elsevier},
  title     = {Texture analysis in materials science},
  year      = {1982},
  address   = {Burlington},
  isbn      = {9780408106429},
  doi       = {10.1016/c2013-0-11769-2},
  pagetotal = {614},
  ppn_gvk   = {804329435},
  subtitle  = {Mathematical Methods},
}

@Article{Haagerup2013,
  author    = {Haagerup, Uffe and Schlichtkrull, Henrik},
  journal   = {The Ramanujan Journal},
  title     = {Inequalities for {J}acobi polynomials},
  year      = {2013},
  issn      = {1572-9303},
  month     = jul,
  number    = {2},
  pages     = {227--246},
  volume    = {33},
  doi       = {10.1007/s11139-013-9472-4},
  publisher = {Springer Science and Business Media LLC},
}

@Article{Hielscher2007,
  author    = {Bachmann, Florian and Hielscher, Ralf and Schaeben, Helmut},
  journal   = {Solid State Phenomena},
  title     = {Texture Analysis with {MTEX} – Free and Open Source Software Toolbox},
  year      = {2010},
  issn      = {1662-9779},
  month     = Feb,
  pages     = {63--68},
  volume    = {160},
  doi       = {10.4028/www.scientific.net/ssp.160.63},
  publisher = {Trans Tech Publications, Ltd.},
}

@InProceedings{Frigo2021,
  author     = {Frigo, Matteo and Johnson, Steven G.},
  booktitle  = {Proceedings of the 1998 IEEE International Conference on Acoustics, Speech and Signal Processing, ICASSP ’98 (Cat. No.98CH36181)},
  title      = {F{FTW}: an adaptive software architecture for the {FFT}},
  pages      = {1381--1384},
  publisher  = {IEEE},
  series     = {ICASSP-98},
  volume     = {3},
  collection = {ICASSP-98},
  doi        = {10.1109/icassp.1998.681704},
}

@Article{Potts1998,
  author    = {Potts, Daniel and Steidl, Gabriele and Tasche, Manfred},
  journal   = {Linear Algebra and its Applications},
  title     = {Fast and stable algorithms for discrete spherical {F}ourier transforms},
  year      = {1998},
  issn      = {0024-3795},
  month     = May,
  pages     = {433--450},
  volume    = {275-276},
  doi       = {10.1016/s0024-3795(97)10013-1},
  publisher = {Elsevier BV},
}

\end{document}